\documentclass[12pt,a4paper,reqno]{amsart}
\usepackage{amssymb}
\usepackage{amscd}
\usepackage{enumerate}
\usepackage{graphicx}
\usepackage{siunitx}
\usepackage{tikz-cd}
\usepackage{color}
\usetikzlibrary{arrows}
\numberwithin{equation}{section}

\usepackage{mathabx}

\usepackage[tableposition=top]{caption}
\usepackage{booktabs,dcolumn}

\newcommand{\SL}{\mathrm{SL}}

\DeclareFontFamily{OT1}{rsfs}{}
\DeclareFontShape{OT1}{rsfs}{n}{it}{<-> rsfs10}{}
\DeclareMathAlphabet{\mathscr}{OT1}{rsfs}{n}{it}

\theoremstyle{plain}

\newtheorem{theorem}{Theorem}[section]
\newtheorem{proposition}[theorem]{Proposition}
\newtheorem{lemma}[theorem]{Lemma}
\newtheorem{corollary}[theorem]{Corollary}

\theoremstyle{definition}

\newtheorem{remark}[theorem]{Remark}

\newtheorem{example}[theorem]{Example}

\newcommand\E{\mathbb{E}}
\newcommand\R{\mathbb{R}}
\newcommand\Z{\mathbb{Z}}
\newcommand\N{\mathbb{N}}

\newcommand\C{\mathbb{C}}
\newcommand\Q{\mathbb{Q}}
\newcommand\eps{\varepsilon}

\renewcommand\P{\mathbb{P}}

\usepackage{algorithm, algorithmic}

\begin{document}

\title{Inverse theorems for sets and measures of polynomial growth}

\author{Terence Tao}
\address{UCLA Department of Mathematics, Los Angeles, CA 90095-1555.}
\email{tao@math.ucla.edu}



\begin{abstract}  We give a structural description of the finite subsets $A$ of an arbitrary group $G$ which obey the polynomial growth condition $|A^n| \leq n^d |A|$ for some bounded $d$ and sufficiently large $n$, showing that such sets are controlled by (a bounded number of translates of) a coset nilprogression in a certain precise sense.  This description recovers some previous results of Breuillard-Green-Tao and Breuillard-Tointon concerning sets of polynomial growth; we are also able to describe the subsequent growth of $|A^m|$ fairly explicitly for $m \geq n$, at least when $A$ is a symmetric neighbourhood of the identity.  We also obtain an analogous description of symmetric probability measures $\mu$ whose $n$-fold convolutions $\mu^{*n}$ obey the condition $\| \mu^{*n} \|_{\ell^2}^{-2} \leq n^d \|\mu \|_{\ell^2}^{-2}$.  In the abelian case, this description recovers the inverse Littlewood-Offord theorem of Nguyen-Vu, and gives a ``symmetrised'' variant of a recent nonabelian inverse Littlewood-Offord theorem of Tiep-Vu.

Our main tool to establish these results is the inverse theorem of Breuillard, Green, and the author that describes the structure of approximate groups.
\end{abstract}

\maketitle


\section{Introduction}

In the field of arithmetic combinatorics, it has turned out to be particularly fruitful to establish \emph{inverse theorems} in which some combinatorial hypothesis on an object of arithmetic combinatorial nature (e.g. a finite subset of a group) is used to establish a much more explicit description of that object.  Ideally, these inverse theorems should be matched as closely as possible by a converse \emph{direct theorem} that shows that all objects of the given explicit description obey the original combinatorial hypothesis (perhaps with some loss in the quantitative constants).  A typical instance of this is Freiman's inverse sumset theorem, first established over the integers by Freiman \cite{freiman}, and in the general setting of abelian groups by Green and Ruzsa \cite{green-ruzsa}:

\begin{theorem}[Freiman's theorem]  Let $G = (G,+)$ be an arbitrary abelian group, and let $A$ be a finite non-empty subset of $G$ of cardinality $|A|$.  Suppose that the sumset $2A = A+A := \{ a_1 + a_2: a_1,a_2 \in A\}$ is such that $|A+A| \leq K |A|$ for some $K \geq 1$.  Then there exists a \emph{coset progression} $H+P$, where $H$ is a finite subgroup of $G$ and $P = P(v_1,\dots,v_r; N_1,\dots,N_r)$ is a \emph{generalised arithmetic progression}, that is to say a set of the form
$$ P = \{ n_1 v_1 + \dots + n_r v_r: n_1,\dots,n_r \in \Z; |n_i| \leq N_i \forall 1 \leq i \leq r \}$$
for some \emph{rank} $r>0$, some $N_1,\dots,N_r>0$ and $v_1,\dots,v_r \in G$, such that $A \subset H+P$, $|H+P| \leq C_K |A|$, and $r \leq C'_K$, where $C_K$, $C'_K$ are quantities depending only on $K$.
\end{theorem}

More explicit values for $C_K$ were given in \cite{green-ruzsa}; for the best known bounds on these quantities\footnote{For the sharpest quantitative results, it is preferable to replace the notion of a coset progression with the slightly more general notion of a \emph{convex coset progression}, and to cover $A$ by a bounded number of translates of $H+P$, rather than just by $H+P$; see \cite{sanders-survey} for more details.} see the survey \cite{sanders-survey} of Sanders.  However, we will not be concerned with quantitative values of constants in this paper.  The corresponding direct theorem is easy: if $H+P$ is a coset progression of some rank $r$ and $A$ is a subset of $H+P$ with $|A| \geq \eps |H+P|$, then it is easy to see that $|A+A| \leq C_{r,\eps} |A|$ for some quantity $C_{r,\eps}$ depending only on $r,\eps$.

A closely related variant of Freiman's theorem concerns iterated sumsets $nA = A + \dots + A$ of $A$ for large $n$, in the case that $nA$ enjoys relative polynomial growth in $n$:

\begin{theorem}[Polynomial growth inverse theorem]\label{pgit-thm}  Let $d \geq 1$, and let $n$ be sufficiently large depending on $d$ (thus $n \geq n_0(d)$ for some $n_0(d)$ depending on $d$). Let $A$ be a finite-non-empty subset of an abelian group $G$ such that $|nA| \leq n^d |A|$.  Then there exists a coset progression $H+P$ of rank at most $C_d$ such that $A \subset H+P$ and $|H+P| \leq C'_d |A|$, where $C_d$ and $C'_d$ depend only on $d$.
\end{theorem}

See for instance \cite[Theorem 2.7]{sanders-survey} for a quantitative version of this theorem with quite good values for the constants $n_0(d), C_d, C'_d$.  Again, the corresponding direct theorem is easy to establish: if $H+P$ has rank at most $r$ and $A \subset H+P$ with $|A| \geq \eps |H+P|$, then $|nA| \leq C_{r,\eps} n^r |A|$ for some $C_{r,\eps}$ depending only on $r,\eps$.

Another family of inverse theorems in arithmetic combinatorics are the \emph{inverse Littlewood-Offord theorems} that classify the tuples of elements $v_1,\dots,v_n$ of an additive group for which the random sum $\pm v_1 \pm \dots \pm v_n$ exhibits unusually large concentration; see \cite{nguyen-survey} for a survey.  A near-optimal such theorem in the case of torsion-free abelian groups is the following result of Nguyen and Vu \cite[Theorem 2.5]{nguyen}:

\begin{theorem}[Inverse Littlewood-Offord theorem]\label{ilot}  Let $v_1,\dots,v_n$ be elements of a torsion-free abelian group $G = (G,+)$, and define the concentration probability
$$ \rho := \sup_{x \in G} \P( \xi_1 v_1 + \dots + \xi_n v_n = x )$$
where $\xi_1,\dots,\xi_n \in \{-1,+1 \}$ are independent Bernoulli variables, each of which attains either sign $+1, -1$ with equal probability.  Suppose that $\rho \geq n^{-A}$ for some $A>0$.  Let $n' \in [n^\eps,n]$ for some $\eps>0$.  Then there exists an arithmetic progression $P$ of rank $r$ at most $C_{\eps,A}$ that contains all but at most $n'$ of the $v_1,\dots,v_n$, such that
$$ |P| \leq C_{\eps,A} \rho^{-1} (n')^{-r/2},$$
where $C_{\eps,A}$ depends only on $\eps$ and $A$.
\end{theorem}

Setting $n'$ to be comparable to $n$ and ignoring the exceptional elements that lie outside of $P$, this result can be matched by a corresponding direct theorem; see \cite[Example 1.7]{nguyen}.  Inverse Littlewood-Offord theorems have a number of applications, particularly to the theory of discrete random matrices; see \cite{nguyen-survey} for details.

In this paper we will be concerned with analogues of these inverse theorems in the context of \emph{non-abelian} groups $G = (G,\cdot)$.  To reflect this we now switch to multiplicative notation rather than additive notation, for instance using product sets $A \cdot B := \{ab: a \in A, b \in B \}$, iterated product sets $A^n = A \cdot \ldots \cdot A$, etc.  

To state the non-abelian version of Freiman's theorem, we need to generalise the notion of a coset progression, recalling some notation from \cite{bgt}.  Define a \emph{nilprogression}\footnote{A pedantic point: technically, a nilprogression is not just the set $P$, but is instead the tuple $(P, r, G, (v_1,\dots,v_r), (N_1,\dots,N_r))$, thus it is possible for distinct nilprogressions to generate the same set $P$ by selecting different choices of $v_1,\dots,v_r$ or $N_1,\dots,N_r$.  Similarly for the concept of a coset nilprogression.  This distinction is needed later when we define the dilations $P^t$, $HP^t$ of a nilprogression $P$ or coset nilprogression $HP$.  However, in the paper we will frequently abuse notation and just refer to the set $P$ (or $HP$) as the nilprogression (or coset nilprogression).} $P = 
P( v_1,\dots,v_r; N_1,\dots,N_r )$, where $v_1,\dots,v_r$ are elements of a group $G$ and $N_1,\dots,N_r>0$ are real numbers, to be the set of all evaluations of words formed from $v_1,\dots,v_r,v_1^{-1},\dots,v_r^{-1}$, where for each $i=1,\dots,r$, the total number of times $v_i$ and $v_i^{-1}$ are used in the word is at most $N_i$, and such that $v_1,\dots,v_r$ generate a nilpotent group of some nilpotency class $s$, which we refer to as the \emph{nilpotency class} of the nilprogression; the quantity $r$ is the \emph{rank} of a nilprogression.  Next, we define a \emph{coset nilprogression} to be a set of the form $HP$, where $H$ is a finite group, and $P$ is a subset of the normaliser $N(H)$ of $H$ which becomes a nilprogression upon applying the quotient map from $N(H)$ to $N(H)/H$, that is to say the coset nilprogression $HP$ is the pullback of a nilprogression in $N(H)/H$.  Note that a coset progression is nothing more than a coset nilprogression in an abelian group.  We define the rank and nilpotency class of a coset nilprogression to be the rank and nilpotency class of the associated nilprogression.  

It turns out to be technically convenient to restrict to\footnote{One could also work instead with the closely related \emph{nilcomplete progressions} and \emph{nilpotent progressions} studied in \cite{tointon}, \cite{bt}.} a well behaved class of nilprogressions and coset nilprogressions.  Following \cite{bgt}, we call a nilprogression $P = P(u_1,\dots,u_r; N_1,\dots,N_r)$ in \emph{$C$-normal form} for some $C \geq 1$ if the following hold:
\begin{itemize}
\item[(i)] (Upper-triangular form) For every $1 \leq i < j \leq r$ and for all four choices of signs $\pm$ one has
$$ [u_i^{\pm 1}, u_j^{\pm 1}] \in P\left( u_{j+1},\dots,u_r; \frac{C N_{j+1}}{N_i N_j}, \dots, \frac{C N_r}{N_i N_j} \right).$$
Here and in the sequel we use the convention $[g,h] := g^{-1}h^{-1}gh$.
\item[(ii)] (Local properness) The expressions $u_1^{n_1} \dots u_r^{n_r}$ are distinct as $n_1,\dots,n_r$ range over the integers with $|n_i| \leq \frac{1}{C} N_i$ for $i=1,\dots,r$.
\item[(iii)] (Volume bound) One has
$$ \frac{1}{C} \prod_{i=1}^r (2\lfloor N_i\rfloor + 1) \leq |P| \leq C \prod_{i=1}^r (2\lfloor N_i\rfloor + 1).$$
\end{itemize}
We say that a coset nilprogression is in $C$-normal form if its associated nilprogression is in $C$-normal form.  Note that for a nilprogression (or coset nilprogression) in $C$-normal form, the nilpotency class is bounded by the rank $r$.

Next, following \cite{tao-product}, we define an \emph{$K$-approximate group} for some $K \geq 1$ to be a subset $A$ of a group $G$ which is symmetric (thus $A^{-1} := \{a^{-1}: a \in A \}$ is equal to $A$), contains the identity $1$, and is such that $A^2$ can be covered by at most $K$ left-translates of $A$.  We then have the following result from \cite[Theorem 2.10]{bgt}, building upon many previous results (see \cite{bgt-survey} for a survey):

\begin{theorem}[Inverse theorem for $K$-approximate groups]\label{bgt-thm}  Let $A$ be a finite $K$-approximate group in an arbitrary group $G$.  Then there exists a coset nilprogression $HP$ of rank and nilpotency class at most $C_K$ in $C_K$-normal form, such that $|HP| \leq C_K |A|$ and that $A$ is covered by at most $C_K$ left-translates of $HP$.  Here $C_K$ is a quantity depending only on $K$.  
\end{theorem}

The bounds on the rank and nilpotency class can be made quite effective; see \cite[Theorem 2.12]{bgt}.  However, the bound on the size of $HP$, the normal form, and the covering number are currently ineffective, due to the reliance in \cite{bgt} on results related to Hilbert's fifth problem.  As we will be relying heavily on Theorem \ref{bgt-thm} in this paper, the bounds in our results are similarly ineffective.  As far as the corresponding direct theorem is concerned, one can check that a coset nilprogression $HP$ of rank $r$ and nilpotency class $s$ in $C$-normal form is a $C_{r,s,C}$-approximate group for some $C_{r,s,C}$ depending on $r,s,C$; see \cite[Lemma C.1, Remark C.2]{bgt} or \cite{bt}.  This direct theorem is not a full converse to the above inverse theorem, since one has to also consider the situation where $A$ is covered by a bounded number of left-translates of $HP$ rather than being equal to $HP$, but shows that the description of approximate groups given in Theorem \ref{bgt-thm} is somewhat close to optimal.

Now we turn to nonabelian analogues of Theorem \ref{pgit-thm}.  We first recall two results in the literature, which were proven as consequences of Theorem \ref{bgt-thm}:

\begin{theorem}\label{exist}  Let $A$ be a symmetric finite subset of a group $G$ containing the identity, let $d > 0$, and suppose that $|A^n| \leq n^d |A|$ for some $n$ that is sufficiently large depending on $d$ (thus $n \geq n_0(d)$ for some sufficiently large $n_0(d)$ depending only on $d$).
\begin{itemize}
\item[(i)] \cite[Theorem 1.13]{bgt}  The group $\langle A \rangle$ generated by $A$ is virtually nilpotent.  Indeed, $\langle A \rangle$ has a subgroup $G'$ of index at most $C_d$, which contains a normal finite subgroup $H$ such that $G'/H$ is nilpotent of rank and nilpotency class at most $C'_d$, where $C_d,C'_d$ depend only on $d$. 
\item[(ii)] \cite[Theorem 1.1]{bt}  For any natural numbers $m \geq n$ and $k>1$, we have $|A^{km}| \leq C_d^k |A^m|$, where $C_d$ depends only on $d$.
\end{itemize}
\end{theorem}

These results give a significant amount of control on $A$; for instance, the first claim (i) easily implies Gromov's theorem \cite{gromov} that finitely generated groups of polynomial growth are virtually nilpotent.  However, these statements are not complemented by a matching direct theorem; the conclusions (i), (ii) do not obviously imply a bound like $|A^n| \leq n^d |A|$.  

\subsection{New results}

Our first main result gives an inverse theorem that comes with a matching direct theorem.  To state it, we need some additional definitions.  Given a nilprogression $P = P(v_1,\dots,v_r;N_1,\dots,N_r)$ in a group $G$ and a parameter $t>0$, we define the dilation $P^t$ by the formula
$$ P^t := P(v_1,\dots,v_r;tN_1,\dots,tN_r).$$
Note that this agrees with the existing definition of $P^t$ as an iterated product when $t$ is a natural number.  Given a group element $g \in G$, we define the ``norm'' $\|g\|_P \in [0,+\infty]$ by the formula
$$ \|g\|_P := \inf \{ t \in [0,+\infty]: g \in P^t \}.$$
It is easy to see that $\|g\|_P=0$ if and only if $g=1$, that $\|g^{-1}\|_P = \|g\|_P$, and $\|gh\|_P \leq \|g\|_P + \|h\|_P$ for all $g,h \in G$, thus justifying the denotation of $\| \|_P$ as a ``norm''.  One can think of the nilprogression $P$ as being analogous to a symmetric convex body in a vector space, in which case $\| \|_P$ is analogous to the norm generated by that body as a unit ball.  Informally, elements $g \in G$ which are small in $\| \|_P$ norm will almost preserve $P$ by left multiplication: $gP \approx P$.

Similarly, given a coset nilprogression $HP$ in a group $G$ and a parameter $t > 0$, we define the coset nilprogression $HP^t$ to be the pullback of the nilprogression $(HP/H)^t$ under the quotient map from $N(H)$ to $N(H)/H$, and then define
$$ \|g\|_{HP} := \inf \{ t \in [0,+\infty]: g \in HP^t \}.$$
Again we have $\|g^{-1}\|_{HP} = \|g\|_{HP}$ and $\|gh\|_{HP} \leq \|g\|_{HP} + \|h\|_{HP}$ for all $g,h \in G$; furthermore one has $\|g\|_P = 0$ if and only if $g \in H$.  Thus one can view $\| \|_{HP}$ as a ``seminorm''.  As before, elements $g \in G$ which are small in $\| \|_{HP}$ seminorm will (informally speaking) almost preserve $HP$: $gHP \approx HP$.

Finally, we need a ``virtual'' extension of the ``seminorm'' $\| \|_{HP}$.  Given a finite non-empty set $X \subset G$ and a group element $g \in G$, we define the ``seminorm'' $\|g\|_{HP,X}$ by the formula
\begin{equation}\label{hps-def}
 \|g\|_{HP,X} := \inf_{\sigma \in \operatorname{Sym}(X)} \sup_{x \in X} \| \sigma(x)^{-1} g x \|_{HP},
\end{equation}
where $\operatorname{Sym}(X)$ denotes the group of permutations $\sigma: X \to X$ on the finite set $X$.
Again, we have $\|g^{-1} \|_{HP,X} = \|g\|_{HP,X}$ and $\|gh\|_{HP,X} \leq \|g\|_{HP,X} + \|h\|_{HP,X}$.  If $\|g\|_{HP,X}=0$ then $gXHP = XHP$; more generally, one should think of elements $g$ that are small in $\| \|_{HP,X}$ seminorm to approximately preserve $XHP$, thus $gXHP \approx XHP$.

\begin{example}\label{sip}  Consider the dihedral group $\{-1,+1\} \ltimes \Z$ of maps $x \mapsto ax+b$ with $a \in \{-1,+1\}$ and $b \in \Z$.  The set of translations $P = \{ x \mapsto x+n: |n| \leq N \}$ is a nilprogression, and hence also a coset nilprogression $HP$ if we take $H$ to be trivial.  If we let $X$ consist of the identity map $x \mapsto x$ and the reflection $x \mapsto -x$, then a map $g$ of the form $x \mapsto ax+b$ with $a \in \{-1,+1\}$ and $b \in \Z$ will have seminorm $\|g\|_{HP,X} = |b|/N$.
\end{example}

We can now state our first main theorem, which we establish in Sections \ref{nonst-sec}-\ref{inv-grow}.

\begin{theorem}[Inverse theorem for polynomial growth]\label{gpg}  Let $A$ be a finite non-empty subset of a group $G$, let $d > 0$, and suppose that $|A^n| \leq n^d |A|$ for some $n$ that is sufficiently large depending on $d$.  Then there exists a coset nilprogression $HP$ of rank and nilpotency class at most $C_d$ in $C_d$-normal form, and a finite subset $X$ of $G$ of cardinality at most $C_d$ containing the identity, such that
\begin{equation}\label{hpn}
 HP \subset (A \cup \{1\} \cup A^{-1})^{C_d n}
\end{equation}
and such that
\begin{equation}\label{sdt}
 A \cup \{1\} \cup A^{-1} \subset\left \{ g \in G: \|g\|_{HP,X} \leq \frac{C_d}{n} \right\},
\end{equation}
where $C_d$ is a natural number depending only on $d$.
\end{theorem}

\begin{example} Continuing Example \ref{sip}, if we take $A$ to be the set of maps $x \mapsto ax+b$ with $a \in \{-1,+1\}$ and $|b| \leq N/n$ for some $1 \leq n \leq N$ then we see that the hypotheses and conclusion of Theorem \ref{gpg} are satisfied with the given value of $HP$ and some absolute constants $d, C_d$.
\end{example}

The conclusion of this theorem gives quite a bit of information about $A$.  For instance, if $m$ is a natural number, we see from \eqref{hpn} that
$$ HP^m \subset (A \cup \{1\} \cup A^{-1})^{C_d mn}$$
and from \eqref{sdt} and the triangle inequality we see that every element of $(A \cup \{1\} \cup A^{-1})^{C_d mn}$ has an $\|\|_{HP,X}$ norm of at most $C_d^2 m$, which in particular implies that such elements lie in $XHP^{C_d^2 m}$ (since $X$ contains the identity).  Thus we have the inclusions
\begin{equation}\label{hom}
HP^m \subset (A \cup \{1\} \cup A^{-1})^{C_d mn} \subset XHP^{C_d^2 m}
\end{equation} 
which show that the growth of $(A \cup \{1\} \cup A^{-1})^{mn}$ is essentially controlled by that of $HP^m$.  (This latter observation was already implicitly used in \cite{bt} to establish Theorem \ref{exist}(ii).)   The $m=1$ case of \eqref{hom} also shows (in the case when $A$ is a symmetric neighbourhood of the identity, so that $A = A \cup \{1\} \cup A^{-1}$) that under the hypotheses \eqref{hpn} and \eqref{sdt}, the condition $|A^n| \leq n^d |A|$ is equivalent (up to constants) to the lower bound $|A| \geq n^{-d} |HP|$, giving a direct theorem to match the inverse theorem.

Using \eqref{hom}, together with a detailed analysis of the growth of coset nilprogressions $HP^m$, we will be able to obtain the following result, which gives a more precise form of Theorem \ref{exist}(ii):

\begin{theorem}[Further growth of a locally polynomially growing set]\label{furt}  Let the notation and hypotheses be as in Theorem \ref{gpg}.  Then there exists a continuous piecewise linear non-decreasing function $f: [0,+\infty) \to [0,+\infty)$, with $f(0)=0$ and $f$ having at most $C_d$ distinct linear pieces, each of which has a slope that is a natural number not exceeding $C_d$, such that
$$ \left|\log |(A \cup \{1\} \cup A^{-1})^{mn}| - \log |(A \cup \{1\} \cup A^{-1})^n| - f(\log m)\right| \leq C_d$$
for all natural numbers $m \geq 1$, where $C_d$ depends only on $d$.
\end{theorem}

We establish this theorem in Section \ref{furt-sec}.  As remarked previously, the inclusions in \eqref{hom} were essentially already obtained in \cite{bt}, so this result should be viewed more as a complement to Theorem \ref{gpg} than a consequence of it.

We illustrate Theorem \ref{furt} with two examples, one which shows that the slope of $f$ can decrease over time, and the other showing that it can increase.  

\begin{example}  Let $G$ be the additive group $G := \Z/N^3\Z \times \Z/N^3\Z$ for a large natural number $N$, and let $A \subset G$ be the set $A := \{-N,\dots,N\} \times \{-N^2,\dots,N^2\}$.  Then one can compute that
$$ \log |mA| = \log |A| + f( \log m ) + O(1) $$
where $f(x)$ is equal to $2x$ for $0 \leq x \leq \log N$, equal to $x + \log N$ for $\log N \leq x \leq 2 \log N$, and equal to $2 \log N$ for $x \geq 2 \log N$.  This gives an example of Theorem \ref{furt} (with $n$ and $d$ equal to suitable constants, independent of $N$) in which the slopes of $f(x)$ decrease as $x$ increases.
\end{example}

\begin{example}  Let $G$ be the Heisenberg group $G = \begin{pmatrix} 1 & \Z & \Z \\ 0 & 1 & \Z \\ 0 & 0 & 1 \end{pmatrix}$, let $N$ be a large natural number, and let $A \subset G$ be the set consisting of those matrices $\begin{pmatrix} 1 & a & b \\ 0 & 1 & c \\ 0 & 0 & 1 \end{pmatrix}$ with $a,b,c \in \Z$, $|a|, |c| \leq N$, and $|b| \leq N^3$.  One can then compute that for any natural number $m$, $(A \cup \{1\} \cup A^{-1})^m$ is commensurate with the set of matrices $\begin{pmatrix} 1 & a & b \\ 0 & 1 & c \\ 0 & 0 & 1 \end{pmatrix}$ with $|a|, |c| \leq mN$ and $|b| \leq m N^3 + m^2 N^2$.  As such one has
$$ \log |(A \cup \{1\} \cup A^{-1})^m| = \log |A| + f( \log m ) + O(1) $$
where $f(x)$ is equal to $3x$ for $0 \leq x \leq \log N$, and equal to $4x - \log N$ for $x \geq \log N$, thus giving an example where the slopes of $f(x)$ increase as $x$ increases.
\end{example}

One can take a Cartesian product of these two examples (with different choices of $N$) to produce an example of a function $f$ which is neither convex nor concave; we leave the verification of this construction to the interested reader.  We conjecture that an analogue of Theorem \ref{furt} holds with $(A \cup \{1\} \cup A^{-1})^{mn}$ replaced by $A^{mn}$, but we will not pursue this question here.  Theorem \ref{furt} may also be compared with the result of Khovanski{\v i} \cite{kh}, which asserts that for any finite subset $A$ of an \emph{abelian} group $G$, the cardinality $|A^n|$ is a polynomial function of $n$ for sufficiently large $n$.  For comparison, Theorem \ref{furt} shows that the function $m \mapsto |A^m|$ is comparable to a piecewise polynomial function of $m$ for $m \geq n$, where the degree and number of pieces of this function, as well as the comparability constants, are bounded by a constant depending only on $d$.  The result in \cite{kh} has been extended to abelian semigroups (see \cite{nat,nathanson}) and to virtually abelian groups (see \cite{benson}); there are partial extensions to the virtually nilpotent case \cite{benson2}, but there exist nilpotent groups for which $|A^n|$ is not eventually polynomial \cite{stoll}, although it is always asymptotic to a polynomial \cite{donne}.

Now we turn to the analogous problem for measures.  Suppose that one is given a discrete probability measure $\mu$ on a group $G = (G,\cdot)$, or equivalently a non-negative function $\mu: G \to \R^+$ with $\sum_{x \in G} \mu(x) = 1$.   The convolution $\mu*\nu: G \to \R^+$ of two such probability measures, defined by
$$ \mu*\nu(x) := \sum_{y \in G} \mu(y) \nu(y^{-1} x)$$
is again a probability measure.  We denote by $\mu^{*n}$ the convolution of $n$ copies of $\mu$.  Note that if $\mu$ is symmetric (by which we mean that $\mu(x^{-1})= \mu(x)$ for all $x \in G$), then $\mu^{*n}$ is symmetric for all $n$.

The quantity $\| \mu \|_{\ell^2(G)}^{-2} := (\sum_{x \in G} \mu(x)^2)^{-1}$, which is a quantity in the interval $[1,+\infty)$, is a measure of how broadly the probability measure $\mu$ is supported.  For instance, if $\mu = \frac{1}{|A|} 1_A$ is uniform measure on a finite set $A$, then $\| \mu \|_{\ell^2(G)}^{-2}$ is equal to $|A|$.  The quantity $\| \mu^{*n} \|_{\ell^2(G)}^{-2}$ is then analogous to the quantity $|A^n|$ discussed previously.  

From Young's inequality we see that the quantity $\| \mu^{*n} \|_{\ell^2(G)}^{-2}$ is non-decreasing in $n$.  The following inverse theorem, analogous to Theorem \ref{gpg}, describes those measures $\mu$ for which $\| \mu^{*n} \|_{\ell^2(G)}^{-2}$ grows at most polynomially in $n$:

\begin{theorem}[Inverse theorem for polynomial growth of measures]\label{gpg-mes}  Let $\mu$ be a symmetric probability measure on a group $G$, let $d > 0$ and $\eps>0$, and suppose that 
\begin{equation}\label{anda}
\| \mu^{*n} \|_{\ell^2(G)}^{-2} \leq n^d \| \mu \|_{\ell^2(G)}^{-2}
\end{equation}
for some $n$ that is sufficiently large depending on $d$ and $\eps$.  Then there exists a coset nilprogression $HP$ of rank and nilpotency class at most $C_{d,\eps}$ in $C_{d,\eps}$-normal form, and a finite subset $X$ of $G$ of cardinality at most $C_{d,\eps}$ containing the identity, such that
\begin{equation}\label{hpn-mes}
 |HP| \leq C_{d,\eps} \| \mu^{*n} \|_{\ell^2(G)}^{-2} 
\end{equation}
and such that
\begin{equation}\label{sdt-mes}
\int_{G \backslash E} \| x\|_{HP,X}^2\ d\mu(x) \leq \frac{C_{d,\eps}}{n}
\end{equation}
for some exceptional set $E$ with
\begin{equation}\label{excep}
\mu(E) \leq \frac{C_{d,\eps}}{n^{1-\eps}}.
\end{equation}
Here $C_{d,\eps}$ is a quantity depending only on $d$ and $\eps$.
\end{theorem}

We prove this theorem in Section \ref{inv-mes-grow}.  Given that Theorem \ref{gpg} does not require symmetry, it is reasonable to expect that some form of Theorem \ref{gpg-mes} can also be established for non-symmetric $\mu$, but we do not pursue this issue here. Some exceptional set $E$ must be permitted in the above theorem; indeed, given a symmetric probability measure $\mu$, one can consider the modified measure $\mu' := (1-\delta) \mu + \delta \nu$ for an arbitrary probability measure $\nu$ and a small $\delta>0$, and then
$$ \| (\mu')^{*n} \|_{\ell^2}^{-2} \leq (1-\delta)^{-2n} \| \mu^{*n} \|_{\ell^2}^{-2}$$
which shows that one can modify $\mu$ more or less arbitrarily on a set of measure $O(1/n)$ without significantly increasing the quantity $\|\mu^{*n}\|_{\ell^2}^{-2}$.  However, the $n^\eps$ loss in \eqref{excep} is somewhat undesirable (as we shall see, it matches the $n^\eps$ loss in Theorem \ref{ilot}), and it may be possible to remove it at the cost of making the conclusion more complicated.

For the corresponding direct theorem, we have the following result, which is a converse to the above theorem if we remove the exceptional set $E$.

\begin{theorem}[Direct theorem for polynomial growth of measures]\label{forward}   Let $\mu$ be a discrete symmetric probability measure on a group $G$.  Let $HP$ be a coset nilprogression in $G$ of rank and nilpotency class bounded by $M$, in $M$-normal form, let $X$ be a non-empty set of cardinality at most $M$, and suppose that
$$
 \int_{g \in G} \|g\|_{HP,X}^2\ d\mu(g) \leq \frac{M}{n}.
$$
Then one has
$$ \| \mu^{*n} \|_{\ell^2}^{-2} \leq C_M |HP |$$
for some quantity $C_M$ depending only on $M$. 
\end{theorem}

We prove this result in Section \ref{forward-proof}; it is essentially a quantitative analysis of random walks on virtually nilpotent groups.  Under additional hypotheses on $X$ and $\mu$, it is possible that one could (in the spirit of \cite{b}) obtain some sort of ``central limit theorem'' that describes $\mu^{*n}$ more precisely, but we do not pursue this issue here.

When the group $G = (G,+)$ is abelian, one can use Theorem \ref{gpg-mes} to recover the inverse Littlewood-Offord theorem in Theorem \ref{ilot}; we do this in Section \ref{ilo-ab}.  

Now we turn to the question of obtaining non-abelian analogues of Littlewood-Offord theory.  This question was recently investigated by Tiep and Vu \cite{tiep}.  We mention just one of their main results:

\begin{theorem}\label{vp}  Let $m,n,s \geq 2$ be integers, and let $A_1,\dots,A_n$ be matrices in $\SL_m(\C)$, each of which has order at least $s$.  Let $\hat A_1,\dots,\hat A_n$ be the independent random matrices selected by choosing $\hat A_i$ to equal $A_i$ or $A_i^{-1}$ with equal probability, for each $i=1,\dots,n$.  Then
$$ \sup_{B \in \SL_m(\C)} \P( \hat A_1 \dots \hat A_n = B) \leq \frac{141}{\min( s, \sqrt{n} )}.$$
\end{theorem}

While our methods cannot recover this type of result exactly, we can obtain the following related result involving a ``symmetrised'' form of the Littlewood-Offord problem, proven in Section \ref{app-sec}:

\begin{theorem}\label{mam}  Let $n \geq 2$ and $0 < \eps \leq 1$.  Let $A_1,\dots,A_n$ be elements of a group $G = (G,\cdot)$, and let $A'_1,\dots,A'_m$ be the independent \emph{identically distributed} random variables with each $A'_i$ selected to equal one of $A_1,\dots,A_n,A_1^{-1},\dots,A_n^{-1}$ with equal probability.  If
\begin{equation}\label{mod}
 \sup_{B \in G} \P( A'_1 \dots A'_n = B) > \frac{1}{\eps \sqrt{n}}
\end{equation}
and $n \geq C/\eps^4$ for a sufficiently large absolute constant $C$, then there exists a finite subgroup $H$ of $G$ of order at most $C \eps \sqrt{n}$ which contains at least $(1-C\eps^2) n$ of the $A_i$.
\end{theorem}

Note if $A_i$ are as in Theorem \ref{vp}, then none of the $A_i$ cannot be contained in any subgroup of order less than $s$, so from Theorem \ref{mam} applied in the contrapositive we see that
$$
 \sup_{B \in G} \P( A'_1 \dots A'_n = B) \leq \frac{C}{\min(s,\sqrt{n})}
$$
for some absolute constant $C$, at least in the regime when $s \geq C n^{1/4}$ (actually one can replace the exponent $1/4$ here by any other positive constant, as can be seen from the argument below).  Thus our theorem differs from Theorem \ref{vp} in that the constant $C$ is not explicit, and we can only control the random walk $A'_1 \dots A'_n$ as opposed to the ordered random product $\hat A_1 \dots \hat A_n$, and one needs some lower bound on $s$.

Let $\mu$ be a finitely supported symmetric probability measure on a group $G$.  By results of Gromov and Varopoulos (see e.g. \cite{woess}), it is known that $\|\mu^{*n}\|_{\ell^\infty(G)}$ decays at a rate $\gg n^{-d/2}$ if and only if $\mu$ is supported in a subgroup $G'$ of $G$ of polynomial volume growth of degree at most $d$ (thus $|S^n| \leq C_S n^d$ for all finite sets $S$ in $G$ and all $n$); see e.g. \cite{woess}.  By modifying the proof of Theorem \ref{mam}, we can obtain\footnote{We thank Emmanuel Breuillard for suggesting this variant.} a non-asymptotic variant of this result:

\begin{theorem}\label{mam-2}  Let $d$ be a natural number and let $\eps>0$.  Let $n$ be a natural number that is sufficiently large depending on $d,\eps$, and let $\mu$ be a symmetric probability measure on a group $G$.  Suppose that
$$ \| \mu^{*n} \|_{\ell^\infty(G)} \geq n^{-(d+1-\eps)/2},$$
then there exists a subgroup $G'$ of $G$ of polynomial volume growth of degree at most $d$ such that $\mu(G') \geq 1-\eps$.
\end{theorem}

We prove this result in Section \ref{app-sec}.

\subsection{Acknowledgments}

The author supported by a Simons Investigator grant, the James and Carol Collins Chair, the Mathematical Analysis \&
Application Research Fund Endowment, and by NSF grant DMS-1266164.  We thank Emmanuel Breuillard, Lam Pham, and Matthew Tointon for helpful comments and corrections, and Emmanuel Breuillard and Melvyn Nathanson for help with the references. We also thank the anonymous referee for many useful comments, corrections, and suggestions.

\subsection{Notation}

We will rely heavily on definitions, notations, and theorems from the paper \cite{bgt}, which also established Theorems \ref{bgt-thm} and \ref{exist}(i) above.  As such, some familiarity with that paper will probably be required in order to easily follow the arguments given here.

\section{Nonstandard analysis formulation}\label{nonst-sec}

To remove some of the ``epsilon management'' in the arguments, and also to more easily access some results from \cite{bgt} that are phrased in a nonstandard setting, we will convert the main results of our paper to a nonstandard formulation.  We will use the nonstandard framework based on a single non-principal ultrafilter $\alpha \in \beta \N \backslash \N$, as laid out in \cite[Appendix A]{bgt}, and will freely use the notation from that appendix in the sequel.  In particular, we have the asymptotic notation $X = O(Y)$, $X \ll Y$, or $X \gg Y$ when $|X| \leq CY$ for some standard $C$, and $X = o(Y)$ if one has $|X| \leq \eps Y$ for every standard $\eps>0$.  We say that a quantity $X$ is \emph{bounded} if $X=O(1)$, and write $X \asymp Y$ for $X \ll Y \ll X$.

We define a \emph{nonstandard group} (or \emph{internal group}) to be an ultraproduct $G = \prod_{{\mathfrak n} \to \alpha} G_{\mathfrak n}$ of (standard) groups $G_{\mathfrak n}$, a \emph{nonstandard finite set} (or \emph{internal finite set}) to be an ultraproduct $A = \prod_{{\mathfrak n} \to \alpha} A_{\mathfrak n}$ of (standard) finite sets $A_{\mathfrak n}$, and so forth.  Note that an internal finite set $A$ has an \emph{internal cardinality} $|A| = \lim_{{\mathfrak n} \to \alpha} |A_{\mathfrak n}|$, which is a nonstandard finite number.

We define an \emph{ultra approximate group} to be an ultraproduct $A = \prod_{{\mathfrak n} \to \alpha} A_{\mathfrak n}$ of (standard) sets $A_{\mathfrak n}|$, which are all $K$-approximate groups for some standard number $K$ independent of ${\mathfrak n}$.  Similarly, define an \emph{ultra coset nilprogression} to be an ultraproduct $HP = \prod_{{\mathfrak n} \to \alpha} H_{\mathfrak n} P_{\mathfrak n}$ of (standard) coset nilprogressions $H_{\mathfrak n} P_{\mathfrak n}$ whose rank and nilpotency class are bounded uniformly in ${\mathfrak n}$; thus $HP$ itself will have rank and nilpotency class which are standard natural numbers.  If $g \in G$, then $\|g\|_{HP}$ is well-defined as a nonstandard element of $[0,+\infty]$; similarly for $\|g\|_{HP,S}$ if $S$ is a nonstandard finite set.  We define the notion of an ultra nilprogression similarly (discarding the finite group $H$).

An ultra nilprogression $P(u_1,\dots,u_r; N_1,\dots,N_r)$ (with $N_1,\dots,N_r$ now nonstandard reals) is said to be in \emph{normal form} if it obeys the following axioms:
\begin{itemize}
\item[(i)] (Upper-triangular form) For every $1 \leq i < j \leq r$ and for all four choices of signs $\pm$ one has
$$ [u_i^{\pm 1}, u_j^{\pm 1}] \in P\left( u_{j+1},\dots,u_r; \frac{O( N_{j+1} )}{N_i N_j}, \dots, \frac{O( N_r )}{N_i N_j} \right).$$
\item[(ii)] (Local properness) The expressions $u_1^{n_1} \dots u_r^{n_r}$ are distinct as $n_1,\dots,n_r$ range over the nonstandard integers with $n_i = o(N_i)$ for $i=1,\dots,r$.
\item[(iii)] (Volume bound) One has
$$ |P| \asymp \prod_{i=1}^r (2\lfloor N_i\rfloor + 1).$$
\end{itemize}
An ultra coset nilprogression is said to be in normal form if its associated ultra nilprogression is in normal form.

Theorem \ref{gpg} then follows from (and is in fact equivalent to) the following nonstandard analysis statement.

\begin{theorem}[Inverse theorem for polynomial growth, nonstandard formulation]\label{gpg-nonst}  Let $A$ be an non-empty internally finite subset of a nonstandard group $G$, let $n$ be an unbounded natural number, and suppose that $|A^n| \leq n^{O(1)} |A|$.  Then there exists an ultra coset nilprogression $HP$ in normal form, and a finite subset $X$ of $G$ of bounded cardinality containing the identity, such that
\begin{equation}\label{hpn-2}
 HP \subset (A \cup \{1\} \cup A^{-1})^{O(n)}
\end{equation}
and such that $\|g\|_{HP,X} = O(1/n)$ for all $g \in A$.
\end{theorem}

Let us see how Theorem \ref{gpg-nonst} implies Theorem \ref{gpg}.  Suppose for contradiction that Theorem \ref{gpg} fails.  Carefully negating the quantifiers, we conclude that there exists $d>0$, a sequence $G_{\mathfrak n}$ of (standard) groups, finite non-empty subsets $A_{\mathfrak n}$ of $G_{\mathfrak n}$, and a sequence $n_{\mathfrak n}$ of natural numbers going to infinity such that
$$ |A_{\mathfrak n}^{n_{\mathfrak n}}| \leq n_{\mathfrak n}^d |A_{\mathfrak n}|$$
for all ${\mathfrak n}$, but such that for each ${\mathfrak n}$, there does \emph{not} exist a coset progression $H_{\mathfrak n}P_{\mathfrak n}$ in $G_{\mathfrak n}$ of rank and nilpotency class at most ${\mathfrak n}$ in ${\mathfrak n}$-normal form and a finite subset $X_{\mathfrak n}$ of $G_{\mathfrak n}$ of cardinality at most ${\mathfrak n}$ containing the identity such that
$$
 H_{\mathfrak n}P_{\mathfrak n} \subset (A_{\mathfrak n} \cup \{1\} \cup A_{\mathfrak n}^{-1})^{{\mathfrak n} n_{\mathfrak n}}
$$
and
$$
 A_{\mathfrak n} \cup \{1\} \cup A_{\mathfrak n}^{-1} \subset\left \{ g_{\mathfrak n} \in G_{\mathfrak n}: \|g_{\mathfrak n}\|_{H_{\mathfrak n}P_{\mathfrak n},X_{\mathfrak n}} \leq \frac{\mathfrak n}{n_{\mathfrak n}} \right\}.
$$
We now take ultraproducts, forming the nonstandard group $G := \prod_{{\mathfrak n} \to \alpha} G_{\mathfrak n}$ and the internally finite subset $A := \prod_{{\mathfrak n} \to \alpha} A_{\mathfrak n}$, and the nonstandard natural number $n := \lim_{{\mathfrak n} \to \alpha} n_{\mathfrak n}$.  By hypothesis, $n$ is unbounded and $|A^n| \leq n^d |A|$.  Thus by Theorem \ref{gpg-nonst}, there exists an ultra coset progression $HP$ in normal form and a finite subset $X$ of $G$ of bounded cardinality containing the identity, such that
$$
 HP \subset (A \cup \{1\} \cup A^{-1})^{C n}
$$
and such that $\|g\|_{HP,X} =\leq C/n$ for all $g \in A$, and some standard $C$.  Writing $HP = \prod_{{\mathfrak n} \to \alpha} H_{\mathfrak n} P_{\mathfrak n}$ and using {\L}os's theorem (see e.g.  \cite[Appendix A]{bgt}), we see that for ${\mathfrak n}$ sufficiently close to $\alpha$ (and enlarging $C$ if necessary), $H_{\mathfrak n} P_{\mathfrak n}$ is a coset nilprogression of rank and nilpotency class at most $C$, in $C$-normal form, with
$$
 H_{\mathfrak n} P_{\mathfrak n} \subset (A_{\mathfrak n} \cup \{1\} \cup A_{\mathfrak n}^{-1})^{C n_{\mathfrak n}}
$$
and such that $\|g_{\mathfrak n}\|_{H_{\mathfrak n}P_{\mathfrak n},X_{\mathfrak n}} =\leq C/n_{\mathfrak n}$ for all $g_{\mathfrak n} \in A_{\mathfrak n}$.  But this contradicts the construction of the $A_{\mathfrak n}$ for ${\mathfrak n}$ large enough.

\section{Inverse theorem for polynomial growth of sets}\label{inv-grow}

We now prove Theorem \ref{gpg-nonst}.  
Let $A$ and $n$ be as in that theorem.  Since $|A^n| \ll n^{O(1)} |A|$ and $n$ is unbounded, we see from the pigeonhole principle that there exists an unbounded $n_0 \ll n$ such that $|A^{100 n_0}| \ll |A^{n_0}|$.  As we will see at the end of the argument, it would be convenient if we could take $n_0 \asymp n$, and from Theorem \ref{exist}(ii) we see that we may do so when $A$ is symmetric and contains the identity; however, we are not assuming symmetry on $A$, and so we will have to temporarily allow for the possibility that $n_0$ is much less than $n$, and return to address this issue at the end of the argument.

From \cite[Corollary 3.11]{tao-product} we see that $(A^{n_0} \cup \{1\} \cup A^{-n_0})^3$ is an ultra approximate group of cardinality $\asymp |A^{n_0}|$.  In particular we have $|(A^{n_0} \cup \{1\} \cup A^{-n_0})^m| \asymp |A^{n_0}|$ for any standard $m$.

If we now applied \cite[Theorem 4.2]{bgt}, we could conclude that the ultra approximate group $(A^{n_0} \cup \{1\} \cup A^{-n_0})^{12}$ contained an ultra coset progression $HP$ in normal form with $|HP| \asymp |A^{n_0}|$.  However it will be convenient to impose an 	 additional ``$\N$-properness'' hypothesis on $P$ that strengthens the local properness property of normal form; this strengthening
is not explicitly provided in \cite[Theorem 4.2]{bgt}, and so we will repeat some of the arguments in \cite{bgt} to obtain this refinement:

\begin{proposition}\label{npr}  There exists an ultra coset nilprogression $HP \subset (A^{n_0} \cup \{1\} \cup A^{-n_0})^{O(1)}$ with $|HP| \asymp |A^{n_0}|$ such that the associated ultra nilprogression $P(\overline{u}_1,\dots,\overline{u}_r; N_1,\dots,N_r) = HP/P$ in normal form obeys the additional property:
\begin{itemize}
\item[(ii')]  ($\N$-properness)  The group elements $u_1^{n_1} \dots u_r^{n_r}$ with $n_1=O(N_1),\dots,n_r = O(N_r)$ are all distinct.
\end{itemize}
\end{proposition}

\begin{proof}  We first apply \cite[Theorem 10.10]{bgt} to conclude the existence of an ultra approximate group $\tilde A \subset(A^{n_0} \cup \{1\} \cup A^{-n_0})^{O(1)}$ with $|\tilde A| \asymp |A^{n_0}|$ which enjoys a \emph{global model} $\phi: \langle \tilde A \rangle \to L$ (as defined before \cite[Proposition 6.10]{bgt}) into a connected, \emph{simply connected} nilpotent Lie group $L$.  Here and in the sequel $\langle \tilde A \rangle$ denotes the \emph{external} group generated by $\tilde A$, that is to say the set of all words in $\tilde A$ of bounded length.  The crucial property here is that $L$ is simply connected; in \cite[Theorem 10.10]{bgt}, this is basically accomplished by quotienting out a maximal compact subgroup from a preliminary Lie model for $\langle A \rangle$.

If one now applies \cite[Theorem 4.2]{bgt} to $\tilde A$, then we see that $\tilde A^4$ contains an ultra coset nilprogression $HP$ in normal form commensurable\footnote{We say that two symmetric sets of a group are commensurable if each set can be covered by a bounded number of left-translates of the other.} with $\tilde A^4$ and hence of cardinality $\asymp |A^{n_0}|$.  Furthermore, an inspection of the proof (using $\phi$ as the Lie model) reveals that $H$ lies in the kernel of $\phi$ (as $L$, being simply connected nilpotent, has no non-trivial compact subgroups), and so by abuse of notation we may also define $\phi$ on the quotient space $\langle HP\rangle/H$; also, the image of $HP$ will be an open neighbourhood of $L$.  Finally, the rank $r$ of $HP$ does not exceed the dimension of the Lie model $L$.  It remains to establish the $\N$-properness, which will ultimately be a consequence of the simply connected nature of $L$.

Suppose for contradiction that we have a collision
$$ \overline{u}_1^{n_1} \dots \overline{u}_r^{n_r} = \overline{u}_1^{n'_1} \dots \overline{u}_r^{n'_r}$$
for some $(n_1,\dots,n_r) \neq (n'_1,\dots,n'_r)$ with $n_i,n'_i = O( N_i)$ for $i=1,\dots,r$.  Cancelling off the $\overline{u}_1$ factors if $n_1=n'_1$, then the $\overline{u}_2$ factors if $(n_1,n_2)=(n'_1,n'_2)$, and so forth, and then moving all copies of the remaining $\overline{u}_i$ to one side using the identity
\begin{equation}\label{gh}
gh = hg [g,h]
\end{equation}
and simplifying using repeated application of the upper triangular form condition (i), we eventually conclude that
$$ \overline{u}_i^{m_i} \dots \overline{u}_r^{m_r} = 1$$
for some $1 \leq i \leq r$ and $m_i,\dots,m_r$ with $m_j = O(N_j)$ for $i \leq j \leq r$ and $m_i$ positive.  Because of this and further repeated application of the upper triangular form condition, we see that for any other $a_1,\dots,a_r$ with $a_j = O(N_j)$ for $1 \leq j \leq r$, we can write
\begin{equation}\label{cough}
 \overline{u}_1^{a_1} \dots \overline{u}_r^{a_r} = \overline{u}_1^{a_1} \dots \overline{u}_{i-1}^{a_{i-1}} \overline{u}_i^{b_i} \dots \overline{u}_r^{b_r}
\end{equation}
where $b_j = O(N_j)$ for $i \leq j \leq r$ and $0 \leq b_i < m_i$.  Meanwhile, by the arguments in \cite[\S 9]{bgt}, we see that for each $1 \leq j \leq r$ there is a one-parameter subgroup $t \mapsto \exp(tX_j)$ in $L$ such that
$$ \phi( \overline{u}_j^{n_j} ) = \exp\left( (\operatorname{st} \frac{n_j}{N_j}) X_j \right)$$
for all $n_j = O(N_j)$.  From this, the normal form property (i) composed with $\phi$, and the simply connected nilpotent nature of $L$ (which makes the exponential map a diffeomorphism, and the Baker-Campbell-Hausdorff formula globally valid) we conclude in particular that $[X_j,X_k]$ lies in the linear span of $X_{k+1},\dots,X_r$ in the Lie algebra of $L$ for all $1 \leq j < k \leq r$.  From this and \eqref{cough} we conclude that every element in $\phi(\langle HP \rangle)$ takes the form
\begin{equation}\label{expt}
 \exp( t_1 X_1 ) \dots \exp( t_r X_r )
\end{equation}
with $t_j \in \R$ for all $1 \leq j \leq r$, with $0 \leq t_i \leq \operatorname{st} \frac{m_i}{N_i}$.  
As $\phi(HP)$ is an open neighbourhood of $L$ and $L$ is connected, $\phi$ must be surjective on $\langle HP \rangle$.
Since $r$ cannot exceed the dimension of $L$, this forces $r$ to in fact be equal to the dimension of $L$, and the $X_1,\dots,X_r$ to be linearly independent.  However, due to the limitation on $t_i$ and the upper triangular nature of the Lie brackets $[X_j,X_k]$ mentioned above, we see that not every element of $L$ is of the form \eqref{expt} (for instance, $\exp( t_i X_i)$ is not of this form if $t_i$ is negative or larger than $\operatorname{st} \frac{m_j}{N_j}$), giving the required contradiction.
\end{proof}

Let $HP$ be as above.  As $(A^{n_0} \cup \{1\} \cup A^{-n_0})^3$ is commensurable with $HP$, it can be covered by a bounded number of left-cosets of the group $\langle HP \rangle$; if we let $\pi: \langle A^{n_0} \rangle \to \langle A^{n_0} \rangle / \langle HP \rangle$ be the quotient map, we thus see that $\pi( A^{n_0} )$ has bounded cardinality.  On the other hand $\pi(A^m)$ is non-decreasing in $m$.  By the pigeonhole principle, we may thus find $1 \leq n_1 < n_0$ such that $\pi(A^{n_1}) = \pi(A^{n_1+1})$.  There thus there exists a finite set $X \in \langle A^{n_0} \rangle$ (of bounded cardinality) such that $A^{n_1+1} \subset X \langle HP \rangle$, and such that the cosets $x \langle HP \rangle$ for $x \in X$ are disjoint and all have non-empty intersection with $A^{n_1}$.  Since
$$ A^{n_1} \langle HP \rangle = A^{n_1+1} \langle HP \rangle = X \langle HP \rangle$$
we see that left multiplication by any element $a$ of $A$ preserves $X \langle HP \rangle$, and so for each $a \in A$ there is a unique permutation $\sigma_a: X \to X$ with the property that
\begin{equation}\label{ashp}
 a x \langle HP \rangle = \sigma_a(x) \langle HP \rangle
\end{equation}
for all $x \in X$.

We now apply an argument of Sanders \cite{sanders}.
Since $|(A^{n_0} \cup \{1\} \cup A^{-n_0})^m| \asymp |A^{n_0}|$ for any standard $m$, we have
$$ |A^{10n_0}| \ll |A^{n_0}|.$$
Thus, by the pigeonhole principle, we may find $n_0  < n_2 < 2n_0$ such that
$$ |A^{n_2+1}| \leq \left(1 + O(\frac{1}{n_0})\right) |A^{n_2}|,$$
thus if we write $B := A^{n_2}$ then
\begin{equation}\label{bab}
 |BA| \leq \left(1 + O(\frac{1}{n_0})\right) |B|.
\end{equation}
We now introduce a nonstandard semi-metric $d: G \times G \to {}^* [0,2]$ on $G$ by the formula
\begin{equation}\label{dsy}
 d( x, y ) := \frac{| Bx \Delta By |}{|B|}
\end{equation}
where $\Delta$ is set-theoretic difference, then $d$ is non-negative, right-invariant (thus $d(xg,yg)=d(x,y)$ for all $g \in G$), symmetric, and obeys the triangle inequality.  From \eqref{bab} (and the fact that $A$ contains the identity) we see that
\begin{equation}\label{dan}
d( a, 1 ) \ll \frac{1}{n_0}
\end{equation}
for all $a \in A$.

To use this, we study the geometry of the semi-metric $d$ on the coset space $\overline{X} := \{ x \langle HP \rangle: x \in X\}$.  We would like to define the nonstandard distances
$$ d_{\overline{X}}( x \langle HP \rangle, x' \langle HP \rangle ) := \inf \{ d( y, y' ): y \in x \langle HP \rangle; y' \in x' \langle HP \rangle \}$$
for any $x,x' \in X$.  Here we run into a technical problem in that the infimum is not automatically defined since the set on the right-hand side is external.  However, we can fix this as follows.  First, using right-invariance, we may (formally) write
$$ d_{\overline{X}}( x \langle HP \rangle, x' \langle HP \rangle) = \inf \{ d( x, y' ): y' \in x' \langle HP \rangle \}.$$
Next, observe that if $d(x,y') < 2$, then from \eqref{dsy} $Bx$ and $By'$ intersect, so $y'$ lies in $B^{-1} Bx$; since
\begin{equation}\label{dor}
\begin{split}
|B^{-1} B x HP| &= |A^{-n_2} A^{n_2} x HP| \\
&\leq |A^{-2n_0} A^{2n_0} X HP| \\
&\leq |(A^{-n_0} \cup \{1\} \cup A^{n_0})^{O(1)}| \\
&\ll |HP|
\end{split}
\end{equation}
we see that $B^{-1} B x$ is covered by boundedly  many left translates of $HP$.  In particular, $B^{-1} Bx \cap x' \langle HP \rangle$ is contained in $x' HP^{m'}$ for some standard $m'$ (which can be made uniform in $x,x'$, since there are only boundedly many choices for these parameters), thus we may (formally) write
$$ d_{\overline{X}}( x \langle HP \rangle, x' \langle HP \rangle) = \inf \{ d( x, y' ): y' \in x' HP^{m'} \} \cup \{2\}.$$
We take this as the \emph{definition} of $d_{\overline{X}}$, and then one easily verifies that the previous two formulae for $d_{\overline{X}}$ are also valid (interpreting the infimum as the greatest lower bound). This makes $d_{\overline{X}}$ a well-defined nonstandard semi-metric on $\overline{X}$.

By construction, we have $d_{\overline{X}}(x \langle HP \rangle,x' \langle HP \rangle) \leq d(x,x')$ for all $x,x' \in X$.  Unfortunately, it is possible for $d_{\overline{X}}(x \langle HP \rangle,x' \langle HP \rangle)$ to be significantly smaller than $d(x,x')$, which will be undesirable for our purposes.  Fortunately, we can fix this by exploiting the ``gauge freedom'' to multiply each $x \in X$ (other than the identity, which we wish to keep in $X$) on the right by an arbitrary element of $\langle HP \rangle$, which does not affect the cosets in $\overline{X}$ or the metric $d_{\overline{X}}$:

\begin{lemma}\label{loma}  After right-multiplying each $x \in X$ (other than the identity) by an element of $\langle HP \rangle$, we can ensure that
$$ d(x,x') \asymp d_{\overline{X}}(x \langle HP \rangle,x' \langle HP \rangle) $$
for all $x,x' \in X$.
\end{lemma}

\begin{proof}  We enumerate $X$ as $x_1,\dots,x_m$ with $x_1=1$, so that cosets in $\overline{X}$ are enumerated as $C_1,\dots,C_m$ with $C_i := x_i \langle HP \rangle$.  We form a spanning tree on these cosets by connecting each $C_i$, $1 < i \leq C_m$, to the $C_j$, $1 \leq j < i$ that minimises the distance $d_{\overline{X}}(C_i,C_j)$; we refer to this $j$ as the ``parent'' of $i$.  (If there is more than one $j$ that minimises the distance, break the tie arbitrarily.)  This is clearly a spanning tree.  We claim the following property: if $1 \leq i < j \leq m$, then $d_{\overline{X}}(C_i,C_j)$ is comparable to the length of the path $C_i = C_{k_1}, C_{k_2}, \dots, C_{k_l} = C_j$ connecting $C_i$ to $C_j$ in the spanning tree, thus
\begin{equation}\label{doe}
 d_{\overline{X}}(C_{k_1},C_{k_2}) + \dots + d_{\overline{X}}(C_{k_{l-1}},C_{k_l}) \asymp d_{\overline{X}}(C_i,C_j).
\end{equation}
The lower bound follows from the triangle inequality.  To prove the upper bound, we assume inductively that the claim has already been established for smaller values of $j$.  If $C_{k_{l-1}} = C_i$ then the claim is trivial, so suppose that $C_{k_{l-1}} \neq C_i$.  As $k_{l-1}$ is the parent of $k_l = j$, we see from construction of the spanning tree we have $d_{\overline{X}}(C_{k_{l-1}},C_{k_l}) \leq d_{\overline{X}}(C_i,C_j)$, so by the triangle inequality $d_{\overline{X}}(C_i, C_{k_{l-1}}) \ll d_{\overline{X}}(C_i,C_j)$.  From the induction hypothesis we have
$$ d_{\overline{X}}(C_{k_1},C_{k_2}) + \dots + d_{\overline{X}}(C_{k_{l-2}},C_{k_{l-1}})  \ll d_{\overline{X}}(C_i, C_{k_{l-1}}) \ll d_{\overline{X}}(C_i,C_j) $$
and the claim follows.

Proceeding recursively from $C_2$ to $C_l$, we may now right-multiply each $x_i$, $2 \leq i \leq l$ by an element of $\langle HP \rangle$ such that
$$ d_{\overline{X}}( C_i, C_{j_i} ) = d( x_i, x_{j_i} )$$
for all $2 \leq i \leq l$, where $j_i$ denotes the parent of $i$. The claim now follows from \eqref{doe} and the triangle inequality.
\end{proof}

Henceforth we let $X$ be chosen to obey the conclusion of the above lemma.  Let $a \in A$ and $x \in X$.  By \eqref{ashp} we have
$$ ax = x' g$$
for some $g \in \langle HP \rangle$ and $x' := \sigma_a(x)$.  From the above lemma, right-invariance, and \eqref{dan}, we have
\begin{align*}
 d(x,x') &\ll d_{\overline{X}}( x \langle HP \rangle, x' \langle HP \rangle ) \\
&\leq d( x, x' g) \\
& = d( x, ax ) \\
&= d( 1, a ) \\
&\ll \frac{1}{n_0}.
\end{align*}
By the triangle inequality and right invariance, we then have
$$ d( x', x' g ) \leq d( x', x ) + d( x, ax ) \ll \frac{1}{n_0} + d(1,a) \ll \frac{1}{n_0}.$$
By repeated application of right-invariance and the triangle inequality we have
$$ d(x', x' g^m) \leq m d(x',x'g)$$
for any nonstandard natural number $m$, and in particular there exists $m_0 \gg n_0$ such that
$$ d(x', x' g^m) < 1 $$
for all $1 \leq m \leq m_0$.  By \eqref{dsy}, this means that $Bx'$ and $Bx' g^m$ intersect for all such $m$, thus
$$ g^m \in (x')^{-1} B^{-1} B x' $$
for all $1 \leq m \leq m_0$.  By \eqref{dor}, the set $(x')^{-1} B^{-1} B x'$ can be covered by boundedly many left-translates of $HP$, and so the orbit $\{ g^m: 1 \leq m \leq m_0\}$ is covered by a bounded number of left translates of $HP$.  Each of these translates either lies in $\langle HP \rangle$ or is disjoint from that group.  On the other hand, $g$ lies in $\langle HP \rangle$.  Thus (by induction on $m$) the only left translates of $HP$ that the orbit $\{ g^m: 1 \leq m \leq m_0\}$ can actually reach are also contained in $\langle HP \rangle$, so that
$$ g^m \in HP^{O(1)}$$
for all $1 \leq m \leq m_0$.  We can then write
$$ g^m \in \overline{u_1}^{n_{1,m}} \dots \overline{u_r}^{n_{r,m}} H$$
for some $n_{i,m}= O(N_i)$ for $i=1,\dots,r$ and $1 \leq m \leq m_0$.  From Proposition \ref{npr} (and the upper triangular form (i)) we conclude that $n_{1,m}$ is linear in $m$, $n_{1,m} = m n_{1,1}$, which implies that 
$$n_{1,1} \ll \frac{N_1}{m} \ll \frac{N_1}{n_0}.$$
A second application of Proposition \ref{npr} (and the upper triangular form) then shows that $n_{2,m}$ is linear in $m$ up to additive errors of $O( m N_2 / n_0 )$, which implies that $n_{2,1} = O( N_2 / n_0 )$.  Continuing in this fashion we see that $n_{i,1} = O(N_i/n_0)$ for all $i$, which implies that $\|g\|_{HP} \ll 1/n_0$, thus
$$ \| \sigma_a(x)^{-1} a x \|_{HP} \ll \frac{1}{n_0}$$
for all $x \in X$, and thus by \eqref{hps-def} we have
\begin{equation}\label{hp}
 \| a\|_{HP,X} \ll \frac{1}{n_0}
\end{equation}
for all $a \in A$.  If we had $n_0 \asymp n$ then we would now have Theorem \ref{gpg-nonst}; however we currently only have $n_0 \ll n$.  We can address this issue as follows.  Firstly, if we set $A' := A \cup \{1\} \cup A^{-1}$, then from \eqref{hp} we have
$$ \| a\|_{HP,X} \ll \frac{1}{n_0}
$$
for all $a \in A'$; also
$$ HP \subset (A^{n_0} \cup \{1\} \cup A^{-n_0})^3 \subset (A')^{3n_0}.$$
By repeating the arguments used to establish \eqref{hom} we see that
\begin{equation}\label{hom-2}
HP^m \subset (A')^{C m n_0} \subset XHP^{C^2 m}
\end{equation} 
for some standard $C$ and all nonstandard natural numbers $m$.  By \cite[(3.2), Proposition 3.10]{bt} we have $|HP^m| \asymp |HP^{m'}|$ whenever $m \asymp m'$, and thus $|(A')^{C m n_0}| \asymp |(A')^{Cm' n_0}|$ whenever $m \asymp m'$.  In particular, we have $|(A')^{100n}| \asymp |(A')^n|$ for any $n \geq n_0$.  If we now repeat the entire argument in this section, and $n_0$ replaced by $n$, until we reach \eqref{hp} again (and with $HP$ replaced by some other ultra coset nilprogression), we obtain Theorem \ref{gpg-nonst} as required.

\section{Further growth of locally polynomial functions}\label{furt-sec}

We now establish Theorem \ref{furt}.  Arguing as in Section \ref{nonst-sec}, we can derive Theorem \ref{furt} from the following nonstandard version:

\begin{theorem}[Further growth of a locally polynomially growing set, nonstandard version]\label{furt-nonst}  
Let $A$ be a symmetric internal finite subset of a nonstandard group $G$ containing the identity, let $n$ be an unbounded natural number, and suppose that $|A^n| \leq n^{O(1)} |A|$.  Then there exists a continuous piecewise linear non-decreasing internal function $f: {}^* [0,+\infty) \to {}^* [0,+\infty)$, with $f(0)=0$ and $f$ having a bounded number of  distinct linear pieces, each of which has a slope that is a standard natural number, such that
$$ \log |A^{mn}| = \log |A^n| + f(\log m) + O(1)$$
for all nonstandard natural numbers $m \geq 1$.
\end{theorem}

Let $A$ and $n$ be as in the above theorem.  By the arguments in the previous section, $A^n$ is commensurable with an ultra coset nilprogression $HP$ in normal form, of some standard rank $r$, which obeys the $\N$-properness property in Proposition \ref{npr}.  We will induct on this parameter $r$, assuming that the claim has already been proven for smaller values of $r$.  From \eqref{hom} and Theorem \ref{exist}(ii) we have
$$ |A^{mn}| \asymp |HP^m| \asymp |H| |\overline{P}^m|$$
where $\overline{P} := HP/H$, which lies in the internal group $N(H)/H$, where $N(H)$ is the normaliser of $H$ in $G$.  Thus matters reduce to showing that
\begin{equation}\label{fam}
\log |\overline{P}^m| = \log |\overline{P}| + f(\log m) + O(1)
\end{equation}
for $f$ as above.

By construction, $\overline{P} = P( \overline{u}_1,\dots,\overline{u}_r; N_1,\dots,N_r)$ is an $\N$-proper ultra nilprogression in normal form.  We may assume that $N_i \gg 1$ for each $i=1,\dots,r$, since any $i$ with $N_i=o(1)$ can simply be deleted from the progression, at which point we can use the induction hypothesis. 
It will be convenient to lift this progression up to a nonstandard simply connected $L$ Lie group (cf. \cite[Lemma C.3]{bgt} for a similar lifting from a local group nilprogression to a global group nilprogression):

\begin{proposition}\label{lak}  There exists a nonstandard simply connected nilpotent Lie group $L$ of dimension $r$ (that is to say, the ultraproduct of standard simply connected nilpotent Lie groups of dimension $r$), a linear basis (over ${}^* \R$) $X_1,\dots,X_r$ of the associated Lie algebra $\log L$ (which can be identified with $L$ using the exponential map $\exp: \log L \to L$ and its inverse, the logarithmic map $\log: L \to \log L$), and a homomorphism $\phi: \Gamma \to N(H)/H$ from the nonstandard group $\Gamma \leq L$ generated by $\exp(X_1),\dots,\exp(X_r)$ to $N(H)/H$ such that $\exp(X_i) = \overline{u}_i$ for all $i=1,\dots,r$.  Furthermore, we have
\begin{equation}\label{xijk}
 [X_i,X_j] = \sum_{i,j < k \leq r} c_{ijk} X_k
\end{equation}
for all $1 \leq i,j \leq r$ and some nonstandard rationals $c_{ijk}$, such that $qc_{ijk}$ is a nonstandard integer for some bounded positive integer $q$, and such that $c_{ijk} = O\left( \frac{N_k}{N_i N_j} \right)$, where $[,]$ of course denotes the Lie bracket on ${\mathfrak L}$.  Finally, for any standard $C>0$, $\phi$ is injective on $P( \exp(X_1),\dots,\exp(X_r); CN_1,\dots,CN_r)$.
\end{proposition}

\begin{proof}  We induct on $r$.  The case $r=0$ is vacuously true.  We could use this as the base case, but the $r=1$ case is also easily verified directly, by setting $L = \log L := {}^* \R$ (with $L$ expressed using additive notation), $X_1 := 1$, and $\phi(n) := \overline{u}_1^n$ for any nonstandard integer $n$.

Now suppose inductively that $r>1$ and that the claim has already been proven for $r-1$.  In particular, there exists a nonstandard simply connected nilpotent Lie group $L_2$ of dimension $r-1$, a linear basis $X_2,\dots,X_r$ of $\log L_2$ over ${}^*\R$, and a homomorphism $\phi_2: \Gamma_2 \to N(H)/H$ from the nonstandard group $\Gamma_2 \leq L_2$ generated by $\exp( X_2), \dots, \exp(X_r)$ to $N(H)/H$ such that $\exp(X_i) = \overline{u}_i$ for $i=2,\dots,r$, and such that \eqref{xijk} holds for $2 \leq i,j \leq r$ and structure constants $c_{ijk}$ with the stated properties, and with $\phi_2$ injective on $P( \exp(X_2),\dots,\exp(X_r); CN_2,\dots,CN_r)$ for any standard $C$.

From the normal form hypothesis on $P$, we have
$$ [\overline{u}_i, \overline{u}_j] = \overline{u}_{j+1}^{n_{i,j,j+1}} \dots \overline{u}_r^{n_{i,j,r}}$$ 
for $1 \leq i < j \leq r$ and some nonstandard integers $n_{i,j,k} = O( \frac{N_k}{N_i N_j} )$.  From the injectivity of $\phi_2$, we conclude that
$$ [\exp( X_i ), \exp( X_j )] = \exp( X_{j+1} )^{n_{i,j,j+1}} \dots \exp(X_r)^{n_{i,j,r}}$$ 
for $2 \leq i < j \leq r$.  It is clear that these relations on the generators $\exp(X_2),\dots,\exp(X_r)$ define $\Gamma_2$ as a group.

Now, the element $\overline{u_1}$ acts on $N(H)/H$ by conjugation $\eta: g \mapsto \overline{u_1}^{-1} g \overline{u}_1 = g [\overline{u}_1, g]^{-1}$.  In particular,
$$ [\eta(\overline{u}_i), \eta(\overline{u}_j)] = \eta(\overline{u}_{j+1})^{n_{i,j,j+1}} \dots \eta(\overline{u}_r)^{n_{i,j,r}}$$ 
for $2 \leq i < j \leq r$.  Observe that $\eta(\overline{u_i}) = \phi_2( \exp(X'_i) )$ for $2 \leq i \leq r$, where
$$ X'_i := \log( \exp( X_i) \exp( X_r)^{-n_{1,i,r}} \dots \exp( X_{i+1} )^{-n_{1,i,j+1}} ).$$
All of the terms in the above identity then lie in $P( \exp(X_2),\dots,\exp(X_r); CN_2,\dots,CN_r)$ for some standard $C$, so by injectivity of $\phi_2$ we conclude that
$$ [\exp( X'_i ), \exp( X'_j )] = \exp( X'_{j+1} )^{n_{i,j,j+1}} \dots \exp(X'_r)^{n_{i,j,r}}$$ 
for $2 \leq i < j \leq r$.  Thus there exists a nonstandard group homomorphism $\tilde \eta: \Gamma_2 \to \Gamma_2$ that maps $\exp(X_i)$ to $\exp(X'_i)$ for $2 \leq i \leq r$.  Applying the same considerations to the inverse conjugation $\eta^{-1}$ we see that $\tilde \eta$ is invertible, and is thus a nonstandard group automorphism on $\Gamma_2$.

We would like to extend this automorphism from the nonstandard discrete group $\Gamma_2$ to the nonstandard Lie group $L_2$.  We first work in the intermediate nonstandard group $\Gamma_2({}^* \Q)$, defined as the set of all elements $g \in L_2$ such that $g^n \in \Gamma_2$ for some nonstandard positive integer $n$.  From the Baker-Campbell-Hausdorff formula (which has only finitely many terms in the nilpotent group $L$) we see that this is a nonstandard subgroup of $L_2$ that contains $\Gamma_2$.  We define the extension $\tilde \eta: \Gamma_2({}^*\Q) \to \Gamma_2({}^*\Q)$ by setting $\tilde \eta(g) := \tilde \eta(g^n)^{1/n}$ for any $g \in \Gamma_2({}^* \Q)$ and any nonstandard positive integer $n$ with $g^n \in \Gamma_2$, where we write $g^t := \exp( t \log g )$ for $g \in L_2$ and $t \in {}^* \R$.  It is easy to see that this extension is well-defined, and from the Baker-Campbell-Hausdorff formula one can verify that it is a nonstandard group homomorphism; applying the same considerations to the inverse of $\tilde \eta$ we see that $\tilde \eta$ is in fact a nonstandard group automorphism on $\Gamma_2({}^* \Q)$.

The group $\Gamma_2({}^* \Q)$ is an internally dense subgroup of $L_2$, and $\tilde \eta$ is internally locally uniformly continuous, so $\tilde \eta$ extends uniquely to a nonstandard continuous group homomorphism on $L_2$, which on consideration of the inverse is in fact a nonstandard continuous group automorphism on $L_2$.  As all continuous homomorphisms between Lie groups are smooth, $\tilde \eta$ is in fact a nonstandard \emph{smooth} group automorphism on $L_2$.  It induces a corresponding nonstandard Lie algebra automorphism $\log \tilde \eta: \log L_2 \to \log L_2$.  By the Baker-Campbell-Hausdorff formula, we see that $\log \tilde \eta(X_i) - X_i$ lies in the linear span of $X_{i+1},\dots,X_r$ over ${}^* \R$ for $1 \leq i \leq r$.

The space of nonstandard Lie algebra automorphisms $\Phi: \log L_2 \to \log L_2$ with the property that $\Phi(X_i) - X_i$ lies in the linear span of $X_{i+1},\dots,X_r$ over ${}^* \R$ for $1 \leq i \leq r$ can be verified to be a nonstandard simply connected nilpotent Lie group.  In particular, we can define an internal one-parameter group $(\log \tilde \eta)^t: \log L_2 \to \log L_2$ in this group for $t \in {}^* \R$ that depend in an internally continuous fashion on $t$, which in turn defines an internal one-parameter group of nonstandard Lie group automorphisms $\tilde \eta^t: L_2 \to L_2$ that also depend internally continuously on $t$.  This lets us define an $r$-dimensional nonstandard Lie group $L := {}^* \R \ltimes_{\tilde \eta} L_2$, which extends $L_2$ by an internal one-parameter group $\{ \exp( t X_1): t \in {}^* \R \}$ such that $\exp( -tX_1) g \exp( t X_1) = \tilde \eta^t(g)$ for all $t \in {}^* \R$ and $g \in L_2$.  
The map $g \mapsto g^{-1} \tilde \eta^t(g)$ is nilpotent from the structure of $\log \tilde \eta^t$ and the Baker-Campbell-Hausdorff formula, so $L$ is nilpotent; as it is internally homeomorphic to ${}^* \R$, it is also internally simply connected.

From construction, we have
$$ 
\exp( - X_1) \exp( X_j) \exp(X_1) = \exp(X'_j)$$
for $2 \leq j \leq r$, which gives \eqref{xijk} for the remaining case $i=1$ from the Baker-Campbell-Hausdorff formula and a downward induction on $j$.  This identity also yields
$$ \phi_2( \exp( - X_1) \exp(X_j) \exp(X_1) ) = \overline{u}_1^{-1} \phi_2( \exp(X_j) ) \overline{u}_1.$$
Since the $\exp(X_j)$ for $2 \leq j \leq r$ internally generate $\Gamma_2$, we conclude that
$$ \phi_2( \exp( - X_1) g \exp(X_1) ) = \overline{u}_1^{-1} \phi_2( g ) \overline{u}_1$$
and thus
$$ \phi_2( \exp( - nX_1) g \exp(nX_1) ) = \overline{u}_1^{-n} \phi_2( g ) \overline{u}_1^n$$
for all $g \in \Gamma_2$ and $n \in {}^* \Z$.  If we then define $\Gamma := {}^* \Z \ltimes \Gamma_2$ to be the internally discrete subgroup of $L$ generated by $\exp(X_1)$ and $\Gamma_2$, we may thus extend the internal homomorphism $\phi_2: \Gamma_2 \to N(H)/H$ to an internal homomorphism $\phi: \Gamma \to N(H)/H$ such that $\phi(\exp(X_1)) = \overline{u}_1$.

Finally, we have to demonstrate the injectivity of $\phi$ on $P( \exp(X_1),\dots,\exp(X_r); O(N_1),\dots,O(N_r))$.  Suppose for contradiction that injectivity failed.  Gathering terms, we obtain a collision of the form
$$ \phi( \exp( n_1 X_1 ) \dots \exp( n_r X_r ) ) = \phi( \exp( n'_1 X_1 ) \dots \exp( n'_r X_r ) ) $$
for some $n_i,n'_i = O(N_i)$ with $(n_1,\dots,n_r) \neq (n'_1,\dots,n'_r)$, and thus
$$ \overline{u}_1^{n_1} \dots \overline{u}_r^{n_r} = \overline{u}_1^{n'_1} \dots \overline{u}_r^{n'_r}.$$
But this contradicts the $\N$-properness of $P( \overline{u}_1,\dots,\overline{u}_r; N_1,\dots,N_r)$.
\end{proof}

\begin{remark} One could also construct the Lie group $L$ here using the theory of Mal'cev bases \cite{malcev}.
\end{remark}

Let $L, \phi, X_1,\dots,X_r$ be as in the above proposition. The set $\overline{P}^m$ can now be expressed as
\begin{equation}\label{pam}
 \overline{P}^m = \phi( Q^m )
\end{equation}
where $Q$ is the nilprogression
$$Q :=  P( \exp(X_1),\dots,\exp(X_r); N_1, \dots, N_r) ).$$
It is thus natural to begin analysing the geometry of $Q^m$.  To do this, we perform some calculations related to those in \cite{bt}, \cite{tointon} (see also the analysis of Carnot-Carath\'eodory balls in \cite{nsw}, \cite{wt} for some analogous calculations).  Define a \emph{formal commutator word} to be any string generated by the following rules:
\begin{itemize}
\item For any $i=1,\dots,r$, $i$ and $i^{-1}$ are formal commutator words.
\item If $w_1, w_2$ are formal commutator words, then the strings $[w_1,w_2]$ and $[w_1,w_2]^{-1}$ are formal commutator words.
\end{itemize}
Thus for instance $[[1,2^{-1}]^{-1},[1^{-1},3]]^{-1}$ will be a formal commutator word if $r \geq 3$.  Define the \emph{length} $|w|$ of a formal commutator word $w$ by requiring $i,i^{-1}$ to have length $1$ for $i=1,\dots,r$, and $[w_1,w_2], [w_1,w_2]^{-1}$ to have length $|w_1|+|w_2|$ for any formal commutator words $w_1,w_2$.  Thus for instance $[[1,2^{-1}]^{-1},[1^{-1},3]]^{-1}$ has length $4$.

Given a formal commutator word $w$, we define the element $X_w$ of $\log L$ as follows.
\begin{itemize}
\item For any $i=1,\dots,r$, we keep $X_i$ as before, and write $X_{i^{-1}} = -X_i$.
\item If $w_1,w_2$ are formal commutator words, then $X_{[w_1,w_2]} = \log( [\exp( X_{w_1}), \exp( X_{w_2}) ] )$ and $X_{[w_1,w_2]^{-1}} = - X_{[w_1,w_2]}$.
\end{itemize}
Thus for instance $X_{[[1,2^{-1}],[1^{-1},3]]^{-1}} = \log [[\exp(X_1),\exp(-X_2)],[\exp(-X_1), \exp(X_3)]]^{-1}$.  Finally, we write $N$ for the vector $N := (N_1,\dots,N_r)$, and define $N^w$ for any formal commutator word by the following rules: 
\begin{itemize}
\item For any $i=1,\dots,r$, we write $N^i = N^{i^{-1}} := N_i$.
\item If $w_1,w_2$ are formal commutator words, then $N^{[w_1,w_2]} = N^{[w_1,w_2]^{-1}} := N^{w_1} N^{w_2}$.
\end{itemize}
Thus for instance $N^{[[1,2^{-1}],[1^{-1},3]]^{-1}} = N_1^2 N_2 N_3$, and for any scalar $m$, $(mN)^{[[1,2^{-1}],[1^{-1},3]]^{-1}} = m^4 N_1^2 N_2 N_3$.  Since $N_i \gg 1$ for all $i$, we have $N^w \gg 1$ for all $w$.  Note that $(mN)^w = m^{|w|} N^w$ for any scalar $m$.

Since $[X_i,X_j]$ lies in the (nonstandard) span of $X_k$ for $k > i,j$, we see that $X_w=0$ for all but a bounded number of words $w$ (for instance, $\overline{u}_w=1$ whenever $|w|>r$).  Let $W$ be the collection of formal commutator words for which $X_w \neq 0$.  We enumerate $W = w_1,w_2,\dots,w_k$ in non-decreasing order of length, starting with the generating words $1,\dots,r$, so in particular $[w_i,w_j]$ lies further along $W$ in this enumeration than $w_i$ or $w_j$, or else fails to lie in $W$ at all.  From the Baker-Campbell-Hausdorff formula we see that for any $1 \leq i < j \leq k$, $[X_{w_i}, X_{w_j}]$ is a linear combination (over the standard rationals $\Q$) of the $X_{w_l}$ for $j < l \leq k$.

We have the following rough description of $P( \exp(X_1),\dots,\exp(X_r); m N_1, \dots, m N_r)$ from \cite{tointon}:

\begin{proposition}\label{mprop}  Let $m$ be a nonstandard natural number.
\begin{itemize}
\item[(i)]  One has
$$ \exp( n_1 X_{w_1} ) \dots \exp( n_k X_{w_k} ) \in Q^{O(m)} $$
whenever $n_1,\dots,n_k$ are nonstandard integers with $n_j = O( (mN)^{w_j} )$ for $j=1,\dots,r$.
\item[(ii)]  Conversely, every element $g$ of $Q^m$ can be written in the form
$$ g = \exp( n_1 X_{w_1} ) \dots \exp( n_k X_{w_k} ) $$
where $n_1,\dots,n_k$ are nonstandard integers with $n_j = O( (mN)^{w_j} )$ for $j=1,\dots,r$.
\end{itemize}
\end{proposition}

\begin{proof}
This follows from \cite[Proposition C.1]{tointon}.  (Strictly speaking, that proposition as written restricts the words under consideration to be basic commutator words, but as observed in \cite{bt}, the argument extends without difficulty to arbitrary words.)  The claim can also be established from \cite[(3.2), Proposition 3.10]{bt}.
\end{proof}

From Proposition \ref{mprop}(ii) and the Baker-Campbell-Hausdorff formula we also see that every element $g$ of $Q^m$ can be written in the form
$$ g = \exp( a_1 X_{w_1} + \dots + a_k X_{w_k} )$$
where $a_1,\dots,a_k$ are nonstandard rationals with standard denominator and $a_j = O( (mN)^{w_j} )$ for $j=1,\dots,r$.  A further application of the Baker-Campbell-Hausdorff formula gives
\begin{equation}\label{bch}
 X_{w_i} = \sum_{j=1}^r \alpha_{i,j} X_j
\end{equation}
for some coefficients $\alpha_{i,j} = O( N_j / N^{w_i} )$ that are nonstandard rationals with standard denominator.  We can therefore write 
\begin{equation}\label{mom}
 g = \exp( \sum_{j=1}^r b_j X_j )
\end{equation}
with $b_j$ nonstandard rationals with standard denominator, and the vector $\vec b := (b_1,\dots,b_r) \in {}^* \R^r$ lying in the (internal) convex hull $B_m$ of $\pm C (mN)^{w_i} \vec \alpha_i$ for $i=1,\dots,k$, where $C$ is a sufficiently large standard quantity and $\vec \alpha_i \in {}^* \R^r$ is the vector $\vec \alpha_i = (\alpha_{i,1},\dots,\alpha_{i,r})$  Since $\vec \alpha_1,\dots,\vec \alpha_r$ is just the canonical basis for ${}^* \R^r$, this convex hull $B_m$ contains the unit cube (if $C$ is large enough), and so from standard volume packing arguments we see that
$$ | Q^m | \ll \operatorname{vol}(B_m).$$
We also have a matching lower bound:

\begin{proposition}\label{lank}  For any nonstandard positive integer $m$, we have
$$ | Q^{m} | \gg \operatorname{vol}(B_m).$$
\end{proposition}

\begin{proof}  Since $B_m$ and $B_{Cm}$ have comparable volume for any standard $C>0$, it will suffice to show that $|Q^{C_0 m}| \gg \operatorname{vol}(B_m)$ for some standard $C_0$.

 By volume packing, we know that $B_m$ contains $\gg \operatorname{vol}(B_m)$ lattice points in ${}^* \Z^r$ (if $C$ is large enough).  Let $(x_1,\dots,x_r)$ be one of these lattice points, and form the Lie algebra vector
$$ v = \sum_{i=1}^r x_i X_i = \sum_{i=1}^k x_k X_{w_i}$$
with $x_i = O( (mN)^{w_i})$ nonstandard real for $i=1,\dots,r$ (in fact they are nonstandard integer), and $x_{r+1}=\dots=x_k=0$.  Our strategy will be to factorise 
$$\exp(v) = \exp( x_1 X_{w_1} + \dots + x_k X_{w_k} )$$
as an element of $Q^{O(m)}$ times a ``bounded'' error.
By rounding $x_1$ to the nearest integer $n_1$ and then using the Baker-Campbell-Hausdorff formula, we may write
$$ \exp(v) = \exp( n_1 X_{w_1} ) \exp( x'_2 X_{w_2} + \dots + x'_k X_{w_k} ) \exp( t_1 X_{w_1} )$$
where $t_1 = O(1)$ is a nonstandard real, $n_1$ is a nonstandard integer with $n_1 = O( (mN)^{w_1} )$, and $x'_j = O( (mN)^{w_j} )$ are nonstandard reals for $j=2,\dots,k$.  Iterating this procedure, we can obtain a factorisation of the form
$$ \exp(v) = \exp( n_1 X_{w_1} ) \dots \exp( n_k X_{w_k} ) \exp(t_k X_{w_k} ) \dots \exp( t_1 X_{w_1} )$$
where $n_i = O( (mN)^{w_i} )$ are nonstandard integers and $t_i = O(1)$ are nonstandard reals for $i=1,\dots,k$.  In particular, by Proposition \ref{mprop}(i) we have
$$ \exp(v) \in Q^{O(m)} \exp(t_k X_{w_k} ) \dots \exp( t_1 X_{w_1} ).$$
Using \eqref{bch} and the Baker-Campbell-Hausdorff formula (discarding the $N^{w_i}$ denominator in the bounds on $\alpha_{i,j}$), we thus have
$$ \exp(v) \in Q^{O(m)} \exp(y_1 X_1 + \dots + y_r X_r )$$
where $y_1,\dots,y_r$ are nonstandard reals with $y_i = O( N_i)$ for $i=1,\dots,r$.  Repeating the previous factorisation procedure, we then have
$$ \exp(v) \in Q^{O(m)} \exp( n'_1 X_1 ) \dots \exp( n'_r X_r) \exp( t'_r X_r ) \dots \exp( t'_1 X_1 )$$
where $n'_i = O(N_i)$ is a nonstandard integer and $t'_i = O(1)$ is a nonstandard real for each $i=1,\dots,r$.  The product $\exp( n'_1 X_1 ) \dots \exp( n'_r X_r)$ lies in $Q^{O(1)}$, hence 
$$ \exp(v) \in Q^{O(m)} \exp( t'_r X_r ) \dots \exp( t'_1 X_1 )$$
or equivalently
$$ \exp( t'_r X_r ) \dots \exp( t'_1 X_1 ) \in Q^{O(m)} \exp(v).$$
By construction, $v$ is a linear combination of $X_1,\dots,X_r$ with coefficients that are nonstandard rational with standard bounded denominator.  By the Baker-Campbell-Hausdorff formula, the same is therefore true of $\log \exp( t'_r X_r ) \dots \exp( t'_1 X_1 )$, which by further application of the Baker-Campbell-Hausdorff formula and induction on $i$ shows that each $t'_i$ is also a nonstandard rational with standard denominator.  Since $t'_i =O(1)$, we conclude that there are only a bounded number of possibilities for each $t'_i$.  We conclude that
$$ \exp(v) \in Q^{O(m)} X$$
for some set $X \subset L$ of bounded cardinality that is independent of $v$.  Since there are $\gg \operatorname{vol}(B_m)$ possibilities for $\exp(v)$, we thus have $|Q^{C_0 m}| \gg \operatorname{vol}(B_m)$ for some sufficiently large $C_0$, as required.
\end{proof}

From Cramer's rule (and selecting the $r$-tuple of vectors of the form $\pm (mN)^{w_i} \vec \alpha_i$ that have the largest wedge product), every element of $B_m$ can be written as a linear combination of $r$ vectors of the form $(mN)^{w_i} \vec \alpha_i$ with nonstandard real coefficients of size $O(1)$.  Conversely, any such linear combination will lie in $B_m$ if the coefficients are smaller than some sufficiently small standard $\eps>0$.  From this we see that
$$ \operatorname{vol}(B_m) \asymp \sum_{1 \leq i_1 < \dots < i_r \leq k} | (mN)^{w_{i_1}} \vec \alpha_{i_1} \wedge \dots \wedge (mN)^{w_{i_r}} \vec \alpha_{i_r} |.$$
The right-hand side is a polynomial $V(m)$ in $m$ of bounded degree and non-negative coefficients.  From the preceding bounds on $Q^m$ we thus have
$$ |Q^m| \asymp V(m).$$
Using $x+y \asymp \max(x,y)$ and the polynomial nature of $V$, we see that $\log V(m)$ is asymptotic to a piecewise linear continuous function of $\log m$, with boundedly many pieces and all slopes non-negative bounded integers.  This gives an estimate of the form
\begin{equation}\label{qa}
 \log |Q^m| = \log |Q| + f(\log m) + O(1)
\end{equation}
with $f$ of the required form (and with the additional property of being convex).  By \eqref{pam}, this establishes the claim in the ``${}^*\N$-proper'' case where $\phi: \Gamma \to N(H)/H$ is injective.  

Now suppose we are in the remaining case when $\phi$ is non-injective.  Then (by the nonstandard well-ordering principle), there exists a nonstandard natural number $m_0$ such that $\phi$ is injective on $Q^m$ if and only if $m < m_0$; by Proposition \ref{lak} we see that $m_0$ is unbounded.  
From \eqref{qa} we see that the desired estimate \eqref{fam} is already established for $m < m_0$; actually from \eqref{qa} and the monotonicity of $\log |\overline{P}^m|$ in $m$, it is established for $m = O(m_0)$.  It remains to handle the case when $m \geq m_0$.  

Since the homomorphism $\phi$ is not injective in $Q^{m_0}$, there exists a non-identity element $g$ of $Q^{O(m_0)}$ such that $\phi(g)=1$.  If $g$ does not commute with every generator $\exp(X_i)$ of $Q$, then we may replace $g$ with the commutator $[g,\exp(X_i)]$ for some $i=1,\dots,r$, which still lies in $Q^{O(m_0)}$ (with a slightly larger implied constant).  Repeating this procedure a bounded number of times, we may assume without loss of generality that $g$ commutes with every $\exp(X_i)$ and thus with $\Gamma$; since $\Gamma$ is internally cocompact in $L$ and the group operations are polynomial, this implies that $g$ is a central element of $L$.

Next, we observe from Theorem \ref{exist}(ii) (or \eqref{qa}) that $|\overline{P}^m| \asymp |\overline{P}^{m_0}|$ for all $m \asymp m_0$, which by \cite[Corollary 3.11]{tao-product} shows that the $\overline{P}^m$ are ultra approximate groups for all $m \asymp m_0$.  As in the proof of Proposition \ref{npr}, we can then find an ultra approximate group $\tilde A \subset \overline{P}^{O(m_0)}$ commensurate with $\overline{P}^{m_0}$ with a global model $\psi: \langle \tilde A \rangle \to L'$ to a connected, simply connected nilpotent Lie group $L'$.

We have a key dimension reduction estimate:

\begin{proposition}  The dimension of $L'$ is strictly smaller than $r$.
\end{proposition}

\begin{proof}  Recall we have the internally convex body $B_{m_0}$ in ${}^* \R^r$.  This generates two external subspaces (over $\R$) of ${}^*\R^r$; the space $O(B_{m_0})$ consisting of all vectors $v$ of the form $\lambda w$ for some $\lambda \in \R$ and $w \in B_{m_0}$, and the subspace $o(B_{m_0})$ be the set of all vectors $v \in {}^*\R^r$ such that $\lambda v \in B_{m_0}$ for all $\lambda \in \R$.  By John's theorem \cite{john}, we see that $o(B_{m_0})$ is an (external) subspace of $O(B_{m_0})$, with a quotient $O(B_{m_0})/o(B_{m_0})$ that has dimension exactly $r$ (over $\R$).  Since $m_0$ is unbounded, $o(B_{m_0})$ contains $\R^r$, and so in particular $O(B_{m_0})/o(B_{m_0})$ is also the image of $O(B_{m_0}) \cap {}^* \Z^r$ under the projection map coming from quotienting by $o(B_{m_0})$.

Pulling back the Lie algebra structure on $\log L$ under the map $\xi: (v_1,\dots,v_r) \mapsto v_1 X_1 + \dots + v_r X_r$, we obtain a nonstandard nilpotent Lie algebra structure on ${}^* \R^r$.  From the definition of $B_{m_0}$, we see that this Lie bracket preserves $O(B_{m_0})$, with $o(B_{m_0})$ as an external Lie algebra ideal (over $\R$).  Thus the vector space $O(B_{m_0})/o(B_{m_0})$ acquires the structure of an $r$-dimensional nilpotent Lie algebra (over $\R$).  Exponentiating this, we obtain an $r$-dimensional nilpotent Lie group $L_0 := \exp( O( B_{m_0}) / o(B_{m_0}) )$ (over $\R$).

Similarly, $\xi( O(B_{m_0}) )$ is an (external) subspace (over $\R$) of $\log L$ that is closed under Lie bracket, with $\xi(o(B_{m_0}))$ an external Lie algebra ideal, so $\exp( \xi(o(B_{m_0})))$ is a normal subgroup of $\exp( \xi(O(B_{m_0})) )$ in $L$, and as $\xi$ is a Lie algebra isomorphism, we see that the quotient group $\exp( \xi(O(B_{m_0}))) / \exp( \xi(o(B_{m_0})) )$ may be identified with $L_0$  and is thus also an $r$-dimensional nilpotent Lie group over $\R$.

Since $\exp( t_i X_i)$ lies in $\exp( \xi(o(B_{m_0})) )$ for any $i=1,\dots,r$ and bounded $t_i$, the same rounding argument used to prove Proposition \ref{lank} shows that any element of $\exp( \xi(O(B_{m_0})))$ may be factored as an element of $\langle Q^{m_0} \rangle$ and an element of $\exp( \xi(o(B_{m_0})) )$.  
In particular, this implies that the image of $Q^{m_0} \subset \exp( \xi(O(B_{m_0})))$ in $L_0$ is a compact neighbourhood of the identity.

Recall that $\tilde A \subset \overline{P}^{Cm_0} =\phi( Q^{Cm_0})$ for some standard natural number $C$.    The set $Q^{Cm_0} \cap \phi^{-1}(\tilde A)$, when projected onto $L_0$, is then a bounded symmetric set $E$, and from countable saturation it is closed.  Since $\tilde A$ is commensurable with $\overline{P}^{m_0} = \phi(Q^{m_0})$, we see that the image of $Q^{m_0}$ in $L_0$ can be covered by a bounded number of translates of $E$, so that $E$ has positive measure.  By the Steinhaus lemma, this implies that $E^2$ contains a neighbourhood of the identity, and in particular generates $L_0$ as a group.  From this we see that to every group element $h \in L_0$, we may find an element $\tilde h$ in the preimage of $h$ in $\langle Q^{m_0} \rangle$ such that $\phi(\tilde h) \in \langle \tilde A \rangle$.  This element $\tilde h$ is defined up to an element of $\langle Q^{m_0} \rangle \cap \exp( \xi(o(B_{m_0})) )$.

Suppose we have two such preimages $\tilde h, \tilde h'$, then $\tilde h' = \tilde h k$ for some $k \in \langle Q^{m_0} \rangle \cap \exp( \xi(o(B_{m_0})) )$ with $\phi(k) \in \langle \tilde A \rangle$.  For any standard natural number $n$, we have $k^n \in \langle Q^{m_0} \rangle \cap \exp( \xi(o(B_{m_0})) )$, so by arguing as in the proof of Proposition \ref{lank}, we see that $k^n \subset Q^{Cm_0} X$ for some standard $C$ and some set $X$ of bounded cardinality (both independent of $n$), so that $\phi(k)^n \subset \overline{P}^{Cm_0} Y$ for some set $Y$ of bounded cardinality.  As $\tilde A$ is commensurate with $\overline{P}^{m_0}$, we conclude that $\phi(k)^n \subset \tilde A Z$ for some set $Z$ of bounded cardinality.  In particular, from the pigeonhole principle we see that $\phi(k)^n \in \tilde A^2$ for infinitely many $n$, and thus $\psi(\phi(k))^n \subset \psi(\tilde A^2)$ for infinitely many $n$.  But $\psi(\tilde A^2)$ is precompact and $L'$ is a simply connected nilpotent Lie group, which (as can be seen by taking logarithms in $L'$) forces $\psi(\phi(k)) = 1$ and hence $\psi(\phi(\tilde h)) = \psi(\phi(\tilde h'))$.  We can thus define a map $\Phi: L_0 \to L'$ such that
$$ \Phi(h) = \psi(\phi(\tilde h))$$
whenever $h \in L_0$ and $\tilde h \in \langle Q^{m_0} \rangle$ is in the preimage of $h$ with $\phi(\tilde h) \in \langle \tilde A \rangle$.  From construction one can verify that $\Phi$ is a group homomorphism, and is continuous at the identity, and is thus a Lie group homomorphism.  Since $\psi: \langle \tilde A \rangle \to L'$ is surjective, we see that $\Phi$ is surjective also.

As $L_0$ has dimension $r$, this already shows that $L'$ has dimension no larger than $r$.  To show that $L'$ has dimension strictly less than $r$,
it will suffice to show that $\Phi$ is not injective.  To do this, suppose for contradiction that $\Phi$ is injective. Recall from construction of $m_0$ that $\phi$ is not injective on $Q^{m_0}$.  Thus, there exists a non-identity element $g$ of $Q^{2m_0}$ that lies in the kernel of $\phi$.  If $g'$ is the projection of $g$ to $L_0$, we thus see that $g'$ lies in the kernel of $\Phi$, and is thus the identity since we are assuming $\phi$ to be injective.  Arguing as in the proof of Proposition \ref{lank}, we conclude that $g^n \in Q^{C m_0} X$ for all standard $n$, some standard $C$, and some $X$ of bounded cardinality.  Since $Q^{Cm_0}$ is commensurate to $Q^{m_0/4}$, we conclude from the pigeonhole principle that $g^n \in Q^{m_0/2}$ for some positive standard integer $n$.  But $g^n$ is not the identity (as can be seen from taking logarithms) and in the kernel of $\phi$, thus $\phi$ is non-injective on $Q^{m_0/2}$, a contradiction.
\end{proof}

Using this global model $\psi$ in the proof of Proposition \ref{npr}, we may thus find an $\N$-proper ultra coset nilprogression $H' P' \subset \overline{P}^{O(m_0)}$ of rank strictly smaller than $r$ with $|H'P'| \asymp |\overline{P}^{m_0}|$.  
From \eqref{hom} and the induction hypothesis, we have
$$ \log |\overline{P}^{m m_0}| = \log |\overline{P}^{m_0}| + f'(\log m) + O(1)$$
for all nonstandard natural numbers $m$ and some piecewise linear continuous function $f$ with boundedly many pieces and all slopes non-negative bounded integers.  By monotonicity of $\log |\overline{P}^m|$ in $m$, this implies that
$$ \log |\overline{P}^{m}| = \log |\overline{P}^{m_0}| + f'(\log m - \log m_0) + O(1)$$
for all $m \geq m_0$.  Concatenating this with the already established case $m = O(m_0)$ of the estimate \eqref{fam}, we obtain \eqref{fam} for all $m$, as required.

\section{Inverse theorem for polynomial growth of measures}\label{inv-mes-grow}

In this section we prove Theorem \ref{gpg-mes}.  

Just as Theorem \ref{gpg} follows from the nonstandard counterpart in Theorem \ref{gpg-nonst}, we may similarly obtain Theorem \ref{gpg-mes} from a nonstandard analogue.  Define a \emph{nonstandard probability measure} $\mu: G \to {}^* \R^+$ on a nonstandard group $G = \prod_{\mathfrak{n} \to \alpha} G_{\mathfrak n}$ to be an ultralimit of standard probability measures $\mu_{\mathfrak n}: G_{\mathfrak n} \to \R^+$.  One can define the (nonstandard) convolution and $\ell^2$ norm of such a probability measure in the obvious fashion, as well as define what it means for a nonstandard probability measure to be symmetric.  Given a nonstandard subset $E$ of $G$, the quantity $\mu(E)$ is then a nonstandard real number between $0$ and $1$, and given a nonstandard function $f: G \to [0,+\infty]$, the integral $\int_G f\ d\mu \in {}^* [0,+\infty]$ is a nonstandard element of $[0,+\infty]$.

By repeating the ``compactness and contradiction'' arguments used to derive Theorem \ref{gpg} from Theorem \ref{gpg-nonst}, we may thus derive Theorem \ref{gpg-mes} from

\begin{theorem}[Inverse theorem for polynomial growth, nonstandard formulation]\label{gpg-mes-nonst}  Let $\mu$ be a symmetric nonstandard probability measure on a nonstandard group $G$, let $\eps>0$ be standard, and suppose that
\begin{equation}\label{anda-nonst}
\| \mu^{*n} \|_{\ell^2(G)}^{-2} \leq n^{O(1)} \| \mu \|_{\ell^2(G)}^{-2}
\end{equation}
for some unbounded natural number $n$.  Then there exists an ultra coset nilprogression $HP$, and a finite subset $X$ of $G$ of bounded cardinality containing the identity, such that
\begin{equation}\label{hpn-mes-nonst}
 |HP| \ll \| \mu^{*n} \|_{\ell^2(G)}^{-2} 
\end{equation}
and such that
\begin{equation}\label{sdt-mes-nonst}
\int_{G \backslash E} \| x\|_{HP,X}^2\ d\mu(x) \ll \frac{1}{n}
\end{equation}
for some exceptional set $E$ with
\begin{equation}\label{excep-nonst}
\mu(E) \ll \frac{1}{n^{1-\eps}}.
\end{equation}
\end{theorem}

We now prove Theorem \ref{gpg-mes-nonst}.  Henceforth we abbreviate $\ell^2(G)$ as $\ell^2$.
Let $B$ be a large standard quantity to be chosen later.
Let $n'$ be the element of $[1,n]$ that maximises the quantity
$$ \| \mu^{*n'} \|_{\ell^2} (n')^B,$$
then we have
\begin{equation}\label{bing}
\| \mu^{*n'} \|_{\ell^2} (n')^B \geq \| \mu^{*n} \|_{\ell^2} n^B \gg n^{B-O(1)} \| \mu \|_{\ell^2}  \geq n^{B-O(1)} \| \mu^{*n'} \|_{\ell^2} 
\end{equation}
and thus
$$ n^{1-O(1/B)} \leq n' \leq n$$
and in particular
$$ n^{1-\eps/2} = o(n')$$
if $B$ is large enough.

Let $n_0 := \lfloor n' / 100 \rfloor$, then $n_0$ is an unbounded natural number with
$$ \| \mu^{*100n_0} \|_{\ell^2} \ll \| \mu^{*n_0} \|_{\ell^2}.$$
(We allow implied constants to depend on $B$ unless otherwise specified.)

Applying the form of the Balog-Szemer\'edi-Gowers theorem from \cite[Proposition A.1]{bggt}, together with Proposition \ref{npr}, we thus have
\begin{equation}\label{munt}
\sup_{x \in G} \mu^{*10n_0}( xHP ) \gg 1
\end{equation}
for some  ultra coset nilprogression $HP$ in normal form with
\begin{equation}\label{hpn-0}
|HP| \ll 1 / \| \mu^{*n_0} \|_{\ell^2}^2.
\end{equation}
Furthermore, the nilprogression $HP$ is infinitely proper in the sense of Proposition \ref{npr}.

By the pigeonhole principle, we may find $n_1$ in $[n_0,2n_0]$ such that
\begin{equation}\label{n11}
 \| \mu^{*n_1+1} \|_{\ell^2} = \left(1 - O\left(\frac{1}{n_0}\right)\right) \| \mu^{*n_1} \|_{\ell^2}
\end{equation}
and
\begin{equation}\label{hpn-2a}
|HP| \ll 1 / \| \mu^{*n_1} \|_{\ell^2}^2.
\end{equation}
From construction we have
\begin{equation}\label{nop}
 n^{1-\eps/2} = o(n_1).
\end{equation}

Squaring \eqref{n11}, we see that
$$ \int_G \int_G \langle \delta_g * \mu^{*n_1}, \delta_h * \mu^{*n_1} \rangle\ d\mu(g) d\mu(h) = \left(1 - O\left(\frac{1}{n_0}\right)\right) \| \mu^{*n_1} \|_{\ell^2}^2$$
and hence by the cosine rule
\begin{equation}\label{bode}
 \int_G \int_G \| \delta_g * \mu^{*n_1} - \delta_h * \mu^{*n_1}\|_{\ell^2}^2 \ d\mu(g) d\mu(h) \ll \frac{1}{n_0} \| \mu^{*n_1} \|_{\ell^2}^2.
\end{equation}
It will be convenient to manipulate this expression a bit.  Taking square roots, and then using the triangle inequality in $h$, we see that
\begin{equation}\label{sod}
 \left(\int_G \| \delta_g * \mu^{*n_1} - \mu^{*n_1+1}\|_{\ell^2}^2 \ d\mu(g)\right)^{1/2} \ll \frac{1}{n_0^{1/2}} \| \mu^{*n_1} \|_{\ell^2}.
\end{equation}
By Young's inequality we also have
\begin{equation}\label{berry}
\left(\int_G \| \delta_g * \mu^{*n_1+1} - \mu^{*n_1+2}\|_{\ell^2}^2 \ d\mu(g)\right)^{1/2} \ll \frac{1}{n_0^{1/2}} \| \mu^{*n_1} \|_{\ell^2}.
\end{equation}
Also, from \eqref{bode} we have
$$ \left(\int_G \int_G \| \delta_h^{-1} * \delta_g * \mu^{*n_1} -  \mu^{*n_1}\|_{\ell^2}^2 \ d\mu(g) d\mu(h)\right)^{1/2} \ll \frac{1}{n_0^{1/2}} \| \mu^{*n_1} \|_{\ell^2}$$
and thus by the triangle inequality in $g,h$ and the symmetry of $\mu$
$$ \left(\int_G \int_G \| \mu^{*n_1+2} - \mu^{*n_1}\|_{\ell^2}^2 \ d\mu(g) d\mu(h)\right)^{1/2} \ll \frac{1}{n_0^{1/2}} \| \mu^{*n_1} \|_{\ell^2}$$
which when combined back with \eqref{berry} and the triangle inequality gives
$$
\left(\int_G \| \delta_g * \mu^{*n_1+1} - \mu^{*n_1+2}\|_{\ell^2}^2 \ d\mu(g)\right)^{1/2} \ll \frac{1}{n_0^{1/2}} \| \mu^{*n_1} \|_{\ell^2}$$
and thus, on combination with \eqref{sod} we have
$$  \left(\int_G \| \delta_g *\nu - \nu \|_{\ell^2}^2 \ d\mu(g)\right)^{1/2} \ll \frac{1}{n_0^{1/2}} \| \mu^{*n_1} \|_{\ell^2} $$
where $\nu := (\mu^{*n_1} + \mu^{*n_1+1}) / 2$; as $\nu$ and $\mu$ are symmetric, we may reverse this as
$$  \left(\int_G \| \nu * \delta_g - \nu \|_{\ell^2}^2 \ d\mu(g)\right)^{1/2} \ll \frac{1}{n_0^{1/2}} \| \mu^{*n_1} \|_{\ell^2} $$

Thus, if we introduce the right-invariant semi-metric
$$ d(g,h) := \| \nu * \delta_g - \nu * \delta_h \|_{\ell^2} / \| \mu^{*n_1} \|_{\ell^2} $$
then we have
\begin{equation}\label{dag}
  \int_G d(g,1)^2 \ d\mu(g) \ll \frac{1}{n_0}.
	\end{equation}

To use this, we exploit the following structural property of a small ball in the $d$ metric:

\begin{lemma} There exists a standard $\eps > 0$ such that the set $\{ g \in G: d(g,1) \leq \eps \}$ is covered by $O(1)$ left translates of $HP^2$. In particular, it is a nonstandard finite set.
\end{lemma}

\begin{proof}  Let $N$ be a standard natural number, and suppose that there exist $g_1,\dots,g_N$ in $\{ g \in G: d(g,1) \leq \eps \}$ with $g_1 HP, \dots, g_N HP$ disjoint.  We will show that $N$ is bounded by a standard constant $C = O(1)$ independent of $N$, which then gives the claim from a greedy argument.

Since $d(g_i,1) \leq \eps$, we have
$$ \| \nu * \delta_{g_i} - \nu \|_{\ell^2} \leq \eps \| \mu^{*n_1} \|_{\ell^2},$$
but from \eqref{munt} we have
$$ \nu( y_0 HP  ) \gg 1$$
for some $y_0 \in G$, and so by Cauchy-Schwarz
$$ \| \nu \|_{\ell^2(y_0 HP)} \gg |HP|^{-1/2} $$
and thus by \eqref{hpn-2a} and translating
$$ \| \nu \|_{\ell^2( y_0 HP  )} \gg \| \mu^{*n_1} \|_{\ell^2}$$
and so if $\eps$ is small enough
$$ \| \nu * \delta_{g_i} \|_{\ell^2( y_0 HP  )} \gg \| \mu^{*n_1} \|_{\ell^2}$$
and thus on summing in $N$ (using the disjointness of $y_0 HP g_i^{-1}$)
$$ \| \nu \|_{\ell^2} \gg N^{1/2} \| \mu^{*n_1} \|_{\ell^2}.$$
But from Young's inequality we have
$$ \| \nu \|_{\ell^2} \leq \| \mu^{*n_1} \|_{\ell^2}$$
and the claim follows.
\end{proof}

We can then cover $\{ g \in G: d(g,1) \leq \eps \}$ by $O(1)$ left cosets of $\langle HP \rangle$.  By the principle of infinite descent, taking $\eps$ sufficiently small, we may assume that $\{ g \in G: d(g,1) \leq \eps \}$ and $\{ g \in G: d(g,1) \leq \eps' \}$ meet exactly the same set $\overline{X} \subset G / \langle HP \rangle$ of left cosets of $\langle HP \rangle$ for any standard $\eps'>0$.

We can put a (nonstandard) quotient metric $d_{\overline{X}}$ on $\overline{X}$ by declaring 
$$ d_{\overline{X}}( x \langle HP \rangle, y \langle HP \rangle ) := \inf \{ d(g,1): g \in G; d(g,1) \leq \eps; g x \langle HP \rangle = y \langle HP \rangle \},$$
noting that the set in the infimum is always a nonstandard finite set.  By the construction of $\eps$, we see that all distances in $d_{\overline{X}}$ are $o(1)$.  Repeating the proof of Lemma \ref{loma}, we may express $\overline{X}$ as $\{ x \langle HP \rangle: x \in X \}$ for some finite set $X$ containing $1$ of bounded cardinality, such that the cosets $x \langle HP \rangle$ for $x \in X$ are all distinct, and
$$ d(x,x') \asymp d_{\overline{X}}( x \langle HP \rangle, x' \langle HP \rangle )$$
for all $x,x' \in X$.  In particular
\begin{equation}\label{story}
X \subset \{ g \in G: d(g,1) = o(1) \} \subset \{ g \in G: d(g,1) \leq \eps \} \subset X HP^C
\end{equation}
for some standard $C$, and that
\begin{equation}\label{dora}
 d(g,1) \gg d(x,x')
\end{equation}
whenever $x,x' \in X$ and $g \in G$ are such that $gx \in x' \langle HP \rangle$.  By right invariance, $d(g,1) = d(gx, x)$, and thus
$$ d( gx, x' ) \ll d(g,1).$$
We can write $gx = x'h$ for some $h \in HP^C$, thus
$$ d(x', x' h ) \ll d(g,1)$$
and thus by right-invariance and the triangle inequality
$$ d(x', x' h^n ) \ll n d(g,1)$$
for any nonstandard $n$.  In particular, by \eqref{story}, $h^n \in HP^C$ whenever $n d(g,1) \leq \eps'$ for some standard $\eps' > 0$.  By the properness of $HP$ as in Section \ref{inv-grow}, we conclude that
$$ \| h \|_{HP} \ll d(g,1),$$
and thus
$$ \| g \|_{HP,X} \ll d(g,1).$$
This claim is clearly also true when $d(g,1) > \eps$.  We conclude from \eqref{dag} that
\begin{equation}\label{mark}
 \int_G \|g\|_{HP,X}^2 \ d\mu(g) \ll \frac{1}{n_0}.
\end{equation}

Now we control the large dimensions of $HP$.  From Markov's inequality and \eqref{mark} we have $\|g\|_{HP,X} = O(1/n_0^{1/2})$ for a nonstandard set of $\mu$-measure $\gg 1$.  In particular, we see that $g \in X HP^{O(1/n_0^{1/2})}$ for such a set.  From the Cauchy-Schwarz inequality, we conclude that
$$ 1 \ll |X HP^{O(1/n_0^{1/2})}|^{1/2} \| \mu \|_{\ell^2},$$
and thus
$$ |HP^{O(1/n_0^{1/2})}| \gg \| \mu \|_{\ell^2}^{-2}.$$
Meanwhile, from \eqref{hpn-0}, \eqref{bing}, we have
\begin{align*}
|HP| &\ll \| \mu^{*n_0} \|_{\ell^2}^{-2} \\
&\ll \| \mu^{*n'} \|_{\ell^2}^{-2} \\
&\ll (n'/n)^B \| \mu^{*n} \|_{\ell^2}^{-2} \\
&\ll (n'/n)^B n^{O(1)} \| \mu \|_{\ell^2}^{-2} \\
&\ll (n')^{O(1)} \| \mu \|_{\ell^2}^{-2} \\
&\ll n_0^{O(1)} \| \mu \|_{\ell^2}^{-2} 
\end{align*}
if $B$ is large enough, where the implied constant in the $O(1)$ exponents do not depend on $B$.  Thus we have
$$ |HP^{O(1/n_0^{1/2})}| \gg n_0^{-O(1)} |HP|.$$
If we let $L_1,\dots,L_r$ be the dimensions of $HP$, we thus have
$$ \prod_{i=1}^r (1 + \frac{L_i}{n_0^{1/2}}) \gg n_0^{-O(1)} \prod_{i=1}^r (1+L_i)$$
and thus $L_i \gg n_0^{\eps/8}$ for at most $O(1)$ values of $i$, where $O(1)$ can depend on $\eps$ but does not depend on $B$.

From \eqref{mark} and Markov's inequality we know that  $\|g\|_{HP,X} < n_0^{-\eps/4}$ for a nonstandard set of $\mu$-measure $1 - O( n_0^{-1+\eps/4} ) = 1 - O(n^{-1+\eps} )$; we denote the complement of this set as $E$.  Let $H\tilde P$ be the coset nilprogression $HP$ with all generators $u_i$ with $N_i = o( n_0^{\eps/8} )$ removed; one easily checks that $H \tilde P$ is still a coset nilprogression in normal form, and from the previous discussion $H\tilde P$ has rank $O(1)$ independently of $B$.  For $g \not \in E$, we have $\|g\|_{HP,X} < n_0^{-\eps/4}$ and hence $\|g\|_{H\tilde P,X} < n_0^{-\eps/4}$ also; from \eqref{mark} we have
\begin{equation}\label{mark-2}
 \int_{G \backslash E} \|g\|_{H\tilde P,X}^2 \ d\mu(g) \ll \frac{1}{n_0}.
\end{equation}
If $n/n_0$ is bounded then we are now done (using $H\tilde P$ in place of $HP$).  Now suppose $n/n_0$ is unbounded.  Let $m$ be the (nonstandard) integer part of $(n/n_0)^{1/2}$, and set $A := H\tilde P^m$.  Using the normal form of $H\tilde P$, one can verify that $H\tilde P^{m_1}$ is covered by a bounded number of translates of $H\tilde P^{m_2}$ whenever $m_1 \ll m_2$.  In particular, $A$ is an ultra approximate group.  We can give $\langle A \rangle$ the structure of a locally compact (but not Hausdorff) group by declaring the neighbourhoods of the identity to be those sets that contain $H\tilde P^{m'}$ for some $m \ll m' \ll m$; this gives a locally compact Hausdorff model for $\langle A \rangle$ after quotienting out by the closure of the identity (i.e. the intersection of all the $H\tilde P^{m'}$ for $m \ll m' \ll m$).  Applying \cite[Theorem 4.2]{bgt} and Proposition \ref{npr}, we see that $A$ is commensurate to an $\N$-proper ultra coset nilprogression $HP'$ in normal form.  Furthermore, if one goes through the proof of \cite[Theorem 4.2]{bgt} starting with the locally compact Hausdorff model provided above, we see that $HP'$ contains an open neighbourhood of the identity in $\langle A \rangle$ (basically because the local Lie models used in the proof are local quotients of the original locally compact model).  Thus we have
$$ H \tilde P^{m_1} \subset HP' \subset H\tilde P^{m_2}$$
for some $m \ll m_1 \ll m_2 \ll m$.  On the one hand, this implies that
$$ \|g\|_{HP',X} \leq \frac{1}{m_1} \|g\|_{H\tilde P,X} \ll \frac{n_0^{1/2}}{n^{1/2}} \|g\|_{H\tilde P,X}$$
whenever $\|g\|_{H\tilde P,X} < 1$.  It also implies that
\begin{align*}
 |HP'| &\leq |H\tilde P^{m_2}| \\
&\ll m_2^{O(1)} |H\tilde P| \\
&\ll (n/n_0)^{O(1)} |HP|
\end{align*}
where the $O(1)$ exponent does not depend on $B$ due to the bounded rank of $H\tilde P$.  For $B$ large enough, \eqref{hpn-0}, \eqref{bing} then gives
$$ |HP'| \ll \|\mu^{*n} \|_{\ell^2}^{-2}$$
and the claim follows (using $HP'$ in place of $HP$).

\section{Direct theorem for polynomial growth of measures}\label{forward-proof}

We now prove Theorem \ref{forward}.  Again, it suffices to establish a nonstandard variant:

\begin{theorem}[Forward Littlewood-Offord theorem]\label{forward-nonst}   Let $G$ be a nonstandard group with nonstandard discrete symmetric probability measure $\mu$.  Let $n$ be an unbounded natural number.  Suppose that there exists a nonstandard coset progression $HP$ in normal form and a non-empty set $X$ of bounded cardinality such that
\begin{equation}\label{moon}
 \int_{g \in G} \|g\|_{HP,X}^2\ d\mu(g) \ll n^{-1}.
\end{equation}
Then one has
$$ \| \mu^{*n} \|_{\ell^2}^{-2} \ll |HP|.$$
\end{theorem}

The derivation of Theorem \ref{forward} from Theorem \ref{forward-nonst} proceeds as in previous sections and is omitted.

Let $\eps>0$ be a small standard quantity to be chosen later.
If we let $\mu'$ be $\mu$ conditioned to the event $\|g\|_{HP,X} < \eps$, then from \eqref{moon} $\mu'$ is a nonstandard symmetric discrete probability measure
with
$$ \mu \geq \left(1 - O(\frac{1}{\eps n})\right) \mu'$$
and thus
$$ \mu^{*n} \gg (\mu')^{*n}$$
with implied constant depending on $\eps$.
Thus we may assume without loss of generality that $\|g\|_{HP,X}<\eps$ on the support of $\mu$, without significantly affecting \eqref{moon}.

Write $\overline{P} = P(\overline{u}_1,\dots,\overline{u}_r; N_1,\dots,N_r)$ denote the nilprogression $\overline{P} := HP/P$.
Let $H\tilde P$ be the ultra coset nilprogression formed from $HP$ by removing those generators $\overline{u}_i$  for which the associated dimension $N_i$ is bounded; as $HP$ is in normal form, it is not difficult to see that $H\tilde P$ is also an ultra coset nilprogression in normal form.  For $\eps>0$ a small enough standard quantity, one has $HP^t=H\tilde P^t$ for all $t<\eps$, and so $\|g\|_{H\tilde P,X}=\|g\|_{HP,X}$ on the support of $\mu$.  Thus, by replacing $HP$ by $H\tilde P$, we may assume without loss of generality that all dimensions $N_i$ are unbounded.

The next step is to represent the coset nilprogression $HP$ locally by a nonstandard Lie group.  By repeated use of the upper triangular property (i), we see that for any nonstandard integers $a_1,\dots,a_r,b_1,\dots,b_r$, we have a multiplication law
$$ \overline{u}_1^{a_1} \dots \overline{u}_r^{a_r} \overline{u}_1^{b_1} \dots \overline{u}_r^{b_r} = 
\overline{u}_1^{P_1(a,b)} \dots \overline{u}_r^{P_r(a,b)}$$
where $a := (a_1,\dots,a_r)$, $b := (b_1,\dots,b_r)$, and $P_1,\dots,P_r: {}^* \Z^r \times {}^* \Z^r \to {}^* \Z$ are polynomials of bounded degree; furthermore, if $a_i,b_i = O(N_i)$ for all $i=1,\dots,r$, then $P_i(a,b) = O(N_i)$ for all $i=1,\dots,r$.

We may uniquely extend the polymomials $P_1,\dots,P_r$ to be polynomials $P_1,\dots,P_r: {}^* \R^r \times {}^* \R^r \to {}^* \R$ of bounded degree; by interpolation (and the unbounded nature of $N_i$) it remains the case that if $a_1,\dots,a_r,b_1,\dots,b_r$ are nonstandard \emph{reals} with $a_i,b_i = O(N_i)$ for all $i=1,\dots,r$, then $P_i(a,b) = O(N_i)$ for all $i=1,\dots,r$.  Since $P_i(0,0)=0$, this implies by further interpolation that if $a_i,b_i = O(\eps N_i)$ for some $0 < \eps < 1$, then $P_i(a,b) = O(\eps N_i)$.  From the local properness of $HP$, we conclude the associativity property
\begin{equation}\label{pabc}
 P( P(a,b), c ) = P( a, P(b,c) ) 
\end{equation}
whenever $a_i,b_i,c_i = O(\eps N_i)$ for $i=1,\dots,r$ and a sufficiently small standard $\eps>0$, where $P: {}^* \R^r \times {}^* \R^r \to {}^* \R^r$ is the polynomial map $P := (P_1,\dots,P_r)$.   Since the $N_i$ are unbounded, this implies from interpolation that the associativity law \eqref{pabc} is in fact valid for \emph{all} $a,b,c \in {}^* \R^r$.  We can thus create a nonstandard Lie group $L$ by setting $L$ to equal the nonstandard vector space ${}^* \R^r$ with multiplication law given by $P$.  From the upper triangular property (i) it is easy to see that this multiplication law is nilpotent, and so $L$ is a nonstandard simply connected Lie group.  In particular, we have a bijective exponential map $\exp: {\mathfrak l} \to L$ from the nonstandard Lie algebra ${\mathfrak l}$ (which as a nonstandard vector space is simply ${}^* \R^r$) to $L$, inverted by a logarithm map $\log: L \to {\mathfrak l}$.  We also have a local representation map $\phi: HP^{C\eps} \to L$ defined (for some suitable large standard $C$ independent of $\eps$) by
$$ \phi(x) := (a_1,\dots,a_r)$$
whenever $x \in HP^{C\eps}$ and $x\ \operatorname{mod}\ H = \overline{u}_1^{a_1} \dots \overline{u}_r^{a_r}$
with $|a_i| \ll \eps N_i$ (with implied constant independent of $\eps$); this is well-defined by local properness if $\eps$ is sufficiently small depending on $C$, and from construction of $L$ we have the local homomorphism property
$$ \phi(xy) = \phi(x) \phi(y)$$
for $x,y \in HP^{C\eps/2}$; also, $H$ lies in the kernel of $\phi$.  Finally, from construction we see that
\begin{equation}\label{sob}
 \| x \|_{HP} \asymp \| \log \phi(x) \|_{\mathfrak l} 
\end{equation}
for $x \in HP^{C\eps}$, where $\|\|_{\mathfrak l}$ is the usual (nonstandard) Euclidean norm on ${\mathfrak l} \equiv {}^* \R^r$.

Let $g$ be chosen at random using the (nonstandard) probability measure $\mu$.  By hypothesis, we can then associate a random permutation $\sigma: X \to X$ and random elements $h_x \in HP$ for $x \in X$ such that
$$ \|h_x \|_{HP} \leq \|g\|_{HP,X}$$
and
$$ gx = \sigma(x) h_x$$
for all $x \in X$.  If there are multiple choices for $\sigma$, we choose amongst them uniformly at random.  The symmetry of $\mu$ then shows that the random tuples
$$ (\sigma, (h_x)_{x \in X})$$
and
$$ (\sigma^{-1}, (h_{\sigma^{-1}(x)}^{-1})_{x \in X})$$
have the same distribution.  From \eqref{moon}, \eqref{sob} we have
\begin{equation}\label{gps}
 \E \|\log \phi(h_s)\|_{\mathfrak l}^2 \ll 1/n
\end{equation}
for all $x \in X$.

For any $i,j \in S$, we let $p_{ij} \in {}^* [0,1]$ denote the (nonstandard) probability that $\sigma(i)=j$, and let $a_{ij} \in {\mathfrak l}$ denote the conditional expectation
\begin{equation}\label{aij}
 a_{ij} = {\mathbf E}(\log(\phi( h_i ))|\sigma(i) = j)
\end{equation}
(with the convention that $a_{ij}=0$ if $p_{ij}=0$).
Observe from the symmetry property mentioned above that $p_{ij}$ is a (nonstandard) symmetric stochastic matrix, in the sense that the $p_{ij}$ are non-negative nonstandard reals with $p_{ij} = p_{ji}$ and
$$ \sum_{j \in X} p_{ij} = 1$$
for all $i \in S$.  Also, we have the crude bound
$$ a_{ij} \ll 1 $$
for all $i,j$, and a further application of symmetry shows that $a_{ij} = -a_{ji}$.  From \eqref{moon} and Cauchy-Schwarz, we also have
$$ a_{ij} \ll p_{ij}^{-1/2} \frac{1}{\sqrt{n}}$$
(this bound is vacuously true if $p_{ij}=0$.

The quantities $a_{ij}$ reflect a certain amount of ``drift'' in the random walk associated to $\mu$.  It would be very convenient if these quantities vanished, or at least obeyed the cancellation condition
$$ \sum_{j \in X} a_{ij} p_{ij} = 0$$
for all $i \in X$.  Such a cancellation need not occur \emph{a priori}.  However it is possible to effectively obtain such a cancellation after a suitable ``gauge change'', which essentially reflects (a continuous version of) the freedom to right-multiply elements of $X$ by small elements in $HP$ (which was already exploited in Sections \ref{inv-grow}, \ref{inv-mes-grow}).  To perform this gauge change, we will need the following linear algebra lemma.

\begin{lemma}\label{stas} Let $P = (p_{ij})_{1 \leq i,j \leq d}$ be a (nonstandard) symmetric stochastic matrix of dimension $d=O(1)$, thus $p_{ij}$ are nonstandard non-negative reals with $p_{ij} = p_{ji}$ and $\sum_{j=1}^d p_{ij} = 1$ for all $i$.  Let $(a_{ij})_{1 \leq i,j \leq d}$ be an anti-symmetric matrix such that $a_{ij} = O(\min( 1, \delta p_{ij}^{-1/2}) )$ for all $i,j$ and some $\delta > 0$.  Then there exist $t_i = O(1)$ for $i=1,\dots,d$ such that
\begin{equation}\label{drift}
 \sum_{j=1}^d a_{ij} p_{ij} = t_i - \sum_{j=1}^d p_{ij} t_j 
\end{equation}
for all $i=1,\dots,d$, and
\begin{equation}\label{sabr}
\sqrt{p_{ij}} (t_i - t_j) = O( \delta )
\end{equation}
for all $i,j$. 
\end{lemma}

\begin{proof} By a compactness argument we may assume that all the $p_{ij}$ are strictly positive, which means that $P$ has a simple eigenvalue at $1$ (with eigenvector consisting of the constant vector with all entries $1/\sqrt{d}$).  Note that the vector $(\sum_{j=1}^d a_{ij} p_{ij})_{i=1}^d$ sums to zero and is thus orthogonal to the constant eigenvector.  If we let $u_2,\dots,u_d$ be an orthonormal eigenbasis of the remaining eigenvalues $\lambda_2,\dots,\lambda_d$, then we have an explicit solution
\begin{equation}\label{ti-exp}
 t_i = \sum_{k=2}^d \frac{1}{1-\lambda_k} (\sum_{l=1}^d \sum_{j=1}^d a_{lj} p_{lj} u_{k,l}) u_{k,i}
\end{equation}
to \eqref{drift}, where $u_{k,i}$ is the $i^{th}$ coefficient of $u_k$.  
We claim the estimate
\begin{equation}\label{plan}
 p_{lj} (u_{k,l} - u_{k,j}) = O( 1 - \lambda_k )
\end{equation}
for all $k,l,j$.  Assuming this estimate, we see from the anti-symmetry of $a_{lj}$ and the bound $a_{lj} = O(1)$ that
$$ \sum_{l=1}^d \sum_{j=1}^d a_{lj} p_{lj} u_{k,l} = O( 1 - \lambda_k )$$
which together with the bounds $u_{k,i} = O(1)$ and \eqref{ti-exp} gives the bounds $t_i = O(1)$.  Next, for any $i,i'$ we have
$$ t_i - t_{i'} = \sum_{k=2}^d \frac{1}{1-\lambda_k} \left(\sum_{l=1}^d \sum_{j=1}^d a_{lj} p_{lj} u_{k,l}\right) (u_{k,i} - u_{k,i'})$$
thanks to \eqref{ti-exp}.  On the one hand we have the trivial bound $u_{k,i} - u_{k,i'} = O(1)$, while from \eqref{plan} we have
$u_{k,i} - u_{k,i'} = O((1-\lambda_k) / p_{lj})$; we take the geometric mean to obtain
\begin{equation}\label{uki}
 u_{k,i} - u_{k,i'} = O((1-\lambda_k)^{1/2} / p_{lj}^{1/2}).
\end{equation}
Similarly, from the anti-symmetry of $a_{l,j}$ and the bound $a_{lj} = O( \delta \sqrt{p_{lj}} )$ we have
$$ \sum_{l=1}^d \sum_{j=1}^d a_{lj} p_{lj} u_{k,l} = O\left( \sum_{l=1}^d \sum_{j=1}^d \delta \sqrt{p_{lj}} |u_{k,l} - u_{k,j}| \right),$$
and hence by \eqref{uki}
$$ \sum_{l=1}^d \sum_{j=1}^d a_{lj} p_{lj} u_{k,l} = O( (1-\lambda_k)^{1/2} \delta ).$$
Putting all this together, we obtain \eqref{sabr}.

It remains to show \eqref{plan}.  The eigenvector equation gives $(1-P) u_k = O( 1- \lambda_k)$, thus
$$ \sum_{j=1}^d p_{ij} (u_{k,i} - u_{k,j}) = O( 1-\lambda_k )$$
for all $i=1,\dots,d$.  If we sort the $u_{k,i}$ in decreasing order,
$$ u_{k,i_{k,1}} \ge u_{k,i_{k,2}} \ge \dots \ge u_{k,i_{k,d}}$$
and apply the above estimate with $i$ equal in turn to $i_{k,1},\dots,i_{k,d}$ and exploit the symmetry of $p_{ij}$, we conclude that
$$ p_{ij} (u_{k,i} - u_{k,j}) = O( 1-\lambda_k )$$
for all $k,i,j$, and the claim follows.
\end{proof}

We can iterate this lemma in the nilpotent Lie group $L$ to obtain

\begin{corollary}\label{jay}  There exists $t_i \in {\mathfrak l}$ for $i \in X$ with
\begin{equation}\label{tii}
t_i = O(1) 
\end{equation}
and
\begin{equation}\label{tib}
 \sqrt{p_{ij}} (t_i - t_j) = O\left( \frac{1}{\sqrt{n}} \right)
\end{equation}
for all $i,j \in X$, such that
$$ \sum_{j \in X} p_{ij} \log( \exp( - t_j) \exp( a_{ij} ) \exp( t_i ) ) = 0 $$
for all $i \in X$.
\end{corollary}

\begin{proof}  Let 
$${\mathfrak l} = {\mathfrak l}_1 \supset {\mathfrak l}_2 \supset {\mathfrak l}_s \supset {\mathfrak l})_{s+1} = \{0\}$$
be the lower central series for the nilpotent Lie algebra ${\mathfrak l}$. We will prove inductively that for any $1 \leq k \leq s+1$, there exist $t_i$ obeying the bounds \eqref{tii}, \eqref{tib} such that
\begin{equation}\label{squash}
\sum_{j \in X} p_{ij} \log( \exp( - t_j) \exp( a_{ij} ) \exp( t_i ) ) \in {\mathfrak l}_k
\end{equation}
for all $i \in X$.  Taking $k=s+1$ will then give the claim.

The case $k=1$ is trivial (just set $t_i=0$ for all $i$), so suppose inductively that the claim has been proven for some $1 \leq k \leq s$, and that we seek to prove the claim for $k+1$.  With $t_i$ as in the induction hypothesis, let $\pi$ be a linear projection from ${\mathfrak l}$ to ${\mathfrak l}_k$ with bounded coefficients, and write
$$ b_{ij} := \pi(  \log( \exp( - t_j) \exp( a_{ij} ) \exp( t_i ) ) ).$$
Then $b_{ij}$ is anti-symmetric.
From the bounds on $a_{ij}, t_i$ and the Baker-Campbell-Hausdorff formula, one has
$$ b_{ij} = O( 1 )$$
and
$$ b_{ij} = O\left( p_{ij}^{-1/2} \frac{1}{\sqrt{n}} \right)$$
for all $i,j$.  By Lemma \ref{stas} applied to the coefficients of $b_{ij}$, we can find $v_i \in {\mathfrak l}_k$ for $i \in X$ such that
$$ v_i = O(1)$$
and
$$
\sqrt{p_{ij}} (v_i - v_j) = O\left( \frac{1}{\sqrt{n}} \right)
$$
for $i,j \in X$, and
$$ \sum_{j=1}^d b_{ij} p_{ij} = t_i - \sum_{j=1}^d p_{ij} t_j $$
for all $i \in X$.  If one then sets
$$ t'_i := t_i + v_i$$
then the inductive claim follows from multiple applications of the Baker-Campbell-Hausdorff formula.
\end{proof}

Let $t_i$ be as in the above lemma.  We now build a ``smooth bump function'' $\Psi: G \to {}^* \R$ by the formula
$$ \Psi(x) := \sum_{s \in X: x \in s HP^{C\eps}} \psi( \eps^{-1} \log( \exp( -t_s ) \phi( s^{-1} x )  ) )$$
for some sufficiently large standard $C$, where $\psi: {\mathfrak l} \to {}^* \R$ is a  non-negative smooth compactly supported function independent of $\eps$ bounded away from zero near the origin, with all derivatives bounded.  

The function $\Psi$ is supported in $XHP^{C\eps}$, is bounded by $O(1)$, and is $\gg 1$ on a set of cardinality $\gg_\eps |HP|$ (where the subscript denotes the fact that the implied constant can depend on $\eps$), and so
\begin{equation}\label{hap}
 \| \Psi \|_{\ell^2} \asymp_\eps |HP|^{1/2}
\end{equation}
and
\begin{equation}\label{hap-1}
 \| \Psi \|_{\ell^1} \asymp_\eps |HP|
\end{equation}
Now we claim that $\Psi$ is essentially stable under convolution by $\mu$:

\begin{lemma}  We have
$$ \| \mu * \Psi - \Psi \|_{\ell^2} \ll_\eps \frac{1}{n} |HP|^{1/2}.$$
\end{lemma}

\begin{proof}  Given that $\mu*\Psi-\Psi$ is supported on $XHP^{O(\eps)}$, it will suffice to show that
$$ \mu * \Psi(sh) - \Psi(sh) = O_\eps( 1/n )$$
for all $s \in X$ and $h \in HP^{O(\eps)}$.

We have
$$ \Psi(s h ) = \psi( \log( \exp( - t_s ) \phi(h) ).$$
Let $g$ be drawn at random using $\mu$, and let the random variables $\sigma$ and $h_s$ be chosen as above.  Then we have
\begin{align*}
\mu * \Psi(sh) &= \E \Psi( g s h ) \\
&= \E \Psi( \sigma(s) h_s h ) \\
&= \E \psi( \eps^{-1} \log( \exp( - t_{\sigma(s)} ) \phi(h_s) \exp(t_s) \exp(-t_s) \phi(h) ) ) 
\end{align*}
and hence by Taylor expansion
$$
\mu * \Psi(sh) - \Psi(sh) = \E \Lambda( \log( \exp( - t_{\sigma(s)} ) \phi(h_s) \exp(t_s) ) ) + O( \E \| \log( \exp( - t_{\sigma(s)} ) \phi(h_s) \exp(t_s) ) \|^2 )
$$
where $\Lambda: {\mathfrak l} \to {}^* \R$ is the derivative of $x \mapsto \psi( \eps^{-1} \log( \exp(x) \exp(-t_s) \phi(h) ) )$ at $x=0$. 
By \eqref{gps}, \eqref{tib} the error term here is $O( \frac{\eps}{n} )$, which is acceptable.   As $\Lambda$ is linear with coefficients $O_\eps(1)$, it suffices to show that
$$ \E \log( \exp( - t_{\sigma(s)} ) \phi(h_s) \exp(t_s) ) = O_\eps( 1 / n ).$$
The left-hand side can be expanded as
$$\sum_{s' \in X} p_{ss'} \E(\log( \exp( - t_{s' }) \phi(h_s) \exp(t_s) ) | \sigma(s)=s' ).$$
By the Baker-Campbell-Hausdorff formula, the expression $\log( \exp( - t_{s' }) \phi(h_s) \exp(t_s) )$ is an affine function of $\log\phi(h_s)$, up to an error of $O_\eps( \| \phi(h_s) \|^2 )$.  By \eqref{aij}, \eqref{gps}, we may thus write this expression as
$$\sum_{s' \in X} p_{ss'} \log( \exp( - t_{s' }) \exp( a_{ss'} ) \exp(t_s) ) + O_\eps(1 / n).$$
By Corollary \ref{jay}, the sum vanishes, and the claim follows.
\end{proof}

Iterating this lemma using Young's inequality and the triangle inequality, we see that
$$ \| \mu^{*\lfloor \delta n\rfloor} * \Psi - \Psi \|_{\ell^2} \ll_\eps \delta |HP|^{1/2}$$
for any $\delta > 0$, and hence that
$$ \langle \mu^{*\lfloor \delta n\rfloor}, \Psi * \tilde \Psi \rangle \gg_\eps \| \Psi \|_{\ell^2}^2$$
for some standard $\delta>0$ depending on $\eps$,
where $\tilde \Psi(g) := \Psi(g^{-1})$.  Since $\Psi * \tilde \Psi$ is supported on $XHP^{O(\eps)}$ and has magnitude $O( |HP| )$, we conclude from \eqref{hap} that
$$ \mu^{*\lfloor \delta n\rfloor}( X HP ) \gg_\eps 1$$
and hence by the pigeonhole principle we have
$$ \mu^{*\lfloor \delta n\rfloor}( s HP ) \gg_\eps 1$$
for some $s \in X$.  Convolving $\mu^{*\delta n}$ with itself (and using symmetry and the identity $(sHP)^{-1} HP = HP^2$) we conclude that
$$ \mu^{*2\lfloor \delta n\rfloor}(HP^2) \gg_\eps 1$$
and thus on further convolution one has
$$ \mu^{*2m\lfloor \delta n\rfloor}(HP^{2m}) \gg_{\eps,m} 1$$
for any standard natural number $m$.  By Cauchy-Schwarz (and the fact that $HP$ is an ultra approximate group, which comes from the normal form and unbounded dimensions of $HP$), this implies that
$$ \| \mu^{*2m\lfloor \delta n\rfloor} \|_{\ell^2}^{-2} \ll_{\eps,m} |HP|$$
and hence by monotonicity (and choosing $m$ large enough depending on $\delta$) one has
$$ \| \mu^{*n} \|_{\ell^2}^{-2} \ll_{\eps,\delta} |HP|$$
which gives the required claim.

\section{Abelian inverse Littlewood-Offord theory}\label{ilo-ab}

We now use Theorem \ref{gpg-mes} to reprove Theorem \ref{ilot}.  The key input is the following Fourier-analytic fact.

\begin{proposition}\label{donk}  Let $G = (G,+)$ be an abelian group.  Let $\mu_1,\dots,\mu_n$ be symmetric discrete probability measures, and define 
$$ \mu := \frac{1}{2} \delta + \frac{1}{2n} (\mu_1 * \mu_1 + \dots + \mu_n * \mu_n),$$
thus $\mu$ is another symmetric probability measure.  Then
$$ \sup_{x \in G} \mu_1 * \dots * \mu_n(\{x\}) \leq \mu^{*n}(\{0\}) \leq \tilde \mu_1 * \dots * \tilde \mu_n(\{0\}) $$
where 
$$ \tilde \mu_j := e^{-1/2} \delta + (1-e^{-1/2}) \mu_j * \mu_j.$$
\end{proposition}

\begin{proof}  By a limiting argument we may assume that $\mu_1,\dots,\mu_n$ are finitely supported, and then by Freiman isomorphism we may take $G$ to be finite.

Standard Fourier analysis shows that
$$ \sup_{x \in G} \mu_1 * \dots * \mu_n(\{x\})$$
is bounded by
$$ \E_{\xi \in \hat G} \prod_{j=1}^n |\hat \mu_j(\xi)|.$$
By the arithmetic mean-geometric mean inequality, we may bound this by
$$ \E_{\xi \in \hat G} (\frac{1}{n} \sum_{j=1}^n |\hat \mu_j(\xi)|)^n.$$
Bounding $|x| \leq \frac{1}{2} + \frac{1}{2} |x|^2$, this is bounded by
$$ \E_{\xi \in \hat G} (\frac{1}{2} + \frac{1}{2n} \sum_{j=1}^n |\hat \mu_j(\xi)|^2)^n$$
which equals
$$ \E_{\xi \in \hat G} \hat \mu(\xi)^n = \mu^{*n}(\{0\}).$$
Next, we bound $x \leq \exp( - (1-x) )$ for $x \leq 1$ to bound this by
$$ \E_{\xi \in \hat G} \exp(-\frac{1}{2n} \sum_{j=1}^n (1-|\hat \mu_j(\xi)|^2))^n$$
which can be rearranged as
$$ \E_{\xi \in \hat G} \prod_{j=1}^n \exp(-\frac{1}{2} (1-|\hat \mu_j(\xi)|^2)).$$
We can bound $\exp( - \frac{1}{2} x ) \leq 1 - (1-e^{-1/2}) x$ for $0 \leq x \leq 1$, so we may bound the above by
$$ \E_{\xi \in \hat G} \prod_{j=1}^n (e^{-1/2} + (1-e^{-1/2}) |\hat \mu_j(\xi)|^2)$$
which equals
$$ \E_{\xi \in \hat G} \prod_{j=1}^n \hat{\tilde \mu_j}(\xi)$$
which equals $\tilde \mu_1 * \dots * \tilde \mu_n(\{0\})$.  The claim follows.
\end{proof}

Now we can prove Theorem \ref{ilot}.  Let $G,n,v_1,\dots,v_n,\rho,\xi_1,\dots,\xi_n,A,n',\eps$ be as in that theorem.  We may assume that $n$ is sufficiently large depending on $\eps,A$, since the claim is trivial otherwise.  For $i=1,\dots,n$, let $\mu_i$ be the Bernoulli probability distribution that takes the values $v_i, -v_i$ with equal probability, then
$$
\rho = \sup_{x \in G} \mu_1 * \dots * \mu_n(\{x\}) $$
and thus by \eqref{donk} one has
$$ \mu^{*n}(\{0\}) \geq \rho \geq n^{-A}$$
and hence by Young's inequality
$$ \| \mu^{* \lfloor n/2\rfloor} \|_{\ell^2(G)}^2 \geq \rho \geq n^{-A}.$$
Applying Theorem \ref{gpg-mes} (with $\eps$ replaced by $\eps/2$, and using additive group notation), we can find a coset progression $H+P$ of rank at most $C_{A,\eps}$, and a finite subset $X$ of $G$ containing the origin of cardinality at most $C_{A,\eps}$, such that
\begin{equation}\label{hpe}
 |H+P| \leq C_{A,\eps} \rho^{-1}
\end{equation}
and such that
$$
\int_{G \backslash E} \| x\|_{H+P,X}^2\ d\mu(x) \leq \frac{C_{A,\eps}}{n}
$$
for some exceptional set $E$ with
$$
\mu(E) \leq \frac{C_{A,\eps}}{n^{1-\eps/2}}.
$$
Here $C_{A,\eps}$ is a quantity depending only on $A$ and $\eps$.

Since $G$ is assumed to be torsion-free, $H$ is trivial.  By Chebyshev's inequality, the hypotheses on $n'$, and the definition of $\mu$, we then have
$$  \| v_i \|_{P,X}^2 \leq \frac{C'_{A,\eps}}{n'}$$
for all but at most $n'$ of the $i \in \{1,\dots,n\}$ and some $C'_{A,\eps}$ depending on $A,\eps$.

Let $i$ be as above, and set $m := \lfloor (n')^{1/2} \rfloor$, then from \eqref{hps-def} there is a permutation $\sigma: X \to X$ such that
$$ \| v_i + x - \sigma(x) \|_P \leq \frac{C''_{A,\eps}}{m}$$
for some $C''_{A,\eps}$ depending on $A,\eps$ and all $x \in X$, which in particular implies upon iteration and telescoping that
$$ \{ v_i, 2v_i, \dots, m v_i\} \subset X + C''_{A,\eps} P.$$
By \cite[Theorem 1.10]{tao-vu}, we can find place $X + C''_{A,\eps} P$ inside an arithmetic progression
$$ P' = \{ n_1 w_1 + \dots + n_r w_r: |n_i| \leq N_i \forall i=1,\dots, r \}$$
of rank $r \leq C'''_{A,\eps}$ and size
$$ |P'| \leq C'''_{A,\eps}  / \rho $$
for some $C'''_{A,\eps}$ depending on $A,\eps$, which is $2$-proper in the sense that the $n_1 w_1 + \dots + n_r w_r$ with $|n_i| \leq 2N_i$ are all distinct.  Since
$$ \{ v_i, 2v_i, \dots, m v_i\} \subset P'$$
we see (by arguing by contradiction) that $v_i$ must lie in the progression
$$ \{ n_1 w_1 + \dots + n_r w_r: |n_i| \leq N_i / m \forall i=1,\dots, r \}.$$
If there are $r'$ choices of $N_i$ for which $N_i \geq m$, then we see that this progression has rank $r'$ and cardinality at most $C''''_{A,\eps} |P'| / m^{r'}$ for some $C''''_{A,\eps}$ depending on $A,\eps$.  Theorem \ref{ilot} follows.

\section{Littlewood-Offord type theorems}\label{app-sec}

In this section we prove Theorems \ref{mam}, \ref{mam-2}.

We first prove Theorem \ref{mam}.  Let $\mu$ be the uniform probability measure on $A_1,\dots,A_n,A_1^{-1},\dots,A_n^{-1}$, then $\mu$ is symmetric and
$$  \sup_{B \in G} \P( A'_1 \dots A'_n = B) = \| \mu^{*n} \|_{\ell^\infty(G)};$$
splitting $\mu^{*n} = \mu^{*n'} * \mu^{*n'} * \mu^{*i}$ with $n' = \lfloor n/2\rfloor$ and $i \in \{0,1\}$ and using Young's inequality, we conclude from \eqref{mod} that
$$ \| \mu^{*n'} \|_{\ell^2(G)}^2 \geq \frac{1}{\eps \sqrt{n}}$$
or equivalently
$$ \| \mu^{*n'} \|_{\ell^2(G)}^{-2} \leq \eps \sqrt{n}.$$
Applying Theorem \ref{gpg-mes} (with $d=1$, $n$ replaced $n'$, and $\eps$ replaced by (say) $1/2$), we obtain a coset nilprogression $HP$ of rank and nilpotency class at most $C$ in $C$-normal form, and a finite subset $X$ of $G$ of cardinality at most $C$ containing the identity, such that
\begin{equation}\label{hpn-mes-a}
 |HP| \leq C \eps \sqrt{n}
\end{equation}
and such that
\begin{equation}\label{sdt-mes-a}
\int_{G \backslash E} \| x\|_{HP,X}^2\ d\mu(x) \leq \frac{C}{n}
\end{equation}
for some exceptional set $E$ with
\begin{equation}\label{excep-a}
\mu(E) \leq \frac{C}{\sqrt{n}} \leq C^{1/2} \eps^2
\end{equation}
for some absolute constant $C$.  From \eqref{sdt-mes-a}, Chebyshev's inequality, and \eqref{excep-a}, we see that we have
$$ \| x \|_{HP,X} < \frac{1}{C \eps \sqrt{n}}$$
for all $x$ outside of a set of $\mu$-measure at most $C^{1/2} \eps^2 + C^3 \eps^2$, thus
$$ \| A_i \|_{HP,X} <\frac{1}{C \eps \sqrt{n}}$$
for all but at most $C' \eps^2 n$ choices of $i=1,\dots,n$, for some absolute constant $C'$.  For each such $i$, we see from \eqref{hps-def} that there exists a permutation $\sigma_i: X \to X$ such that
$$
\| \sigma(x)^{-1} A_i x \|_{HP} < \frac{1}{C \eps \sqrt{n}}$$
for all $x \in X$.  Since $|P| \leq |HP| \leq C \eps \sqrt{n}$, we conclude that
$$ \sigma(x)^{-1} A_i x \in H$$
for all $x \in X$.  Thus $A_i$ lives in the group $H' := \{ g \in G: gXH = XH \}$, which is a group of order at most $|X| |H| \leq C^2 \eps \sqrt{n}$.  The claim follows.

Now we prove Theorem \ref{mam-2}.  By reducing $n$ by one if necessary (and adjusting $\eps$ slightly), we may assume that $n$ is even, then by Young's inequality
$$ \| \mu^{*n/2} \|_{\ell^2(G)}^{-2} \leq n^{d+1-\eps}$$
and so by Theorem \ref{gpg-mes} as in the proof of Theorem \ref{mam}, we can find a coset nilprogression $HP$ in $C_{d,\eps}$-normal form of rank and nilpotency class at most $C_{d,\eps}$ with
\begin{equation}\label{happy}
 |HP| \leq C_{d,\eps} n^{d+1-\eps}
\end{equation}
and a set $X \subset G$ of cardinality at most $C_{d,\eps}$ such that the set
$$ F := \{ g \in G: \|g\|_{HP,X} \leq \frac{C_{d,\eps}}{\sqrt{n}} \}  $$
has $\mu$-measure at least $1-\eps$, where $C_{d,\eps}$ depends only on $d,\eps$.  

Write $HP/H = P(v_1,\dots,v_r; N_1,\dots,N_r)$.  We may delete any generator $v_i$ with $N_i < C_{d,\eps}^{-1} \sqrt{n}$ (without significantly affecting the $C_{d,\eps}$-normal form), since this does not alter $F$.  If $g \in F$, then by \eqref{hps-def} there is a permutation $\sigma: X \to X$ such that $\sigma(x)^{-1} g x \in \langle HP \rangle$ for all $x \in X$, so that $g X\langle HP \rangle = X\langle HP \rangle$.  The group $\{ g: g X \langle HP \rangle = X \langle HP \rangle \}$ contains $\langle HP \rangle$ as a finite index subgroup (since this is a stabiliser of the action of this group on the finite space $X \langle HP \rangle / \langle HP \rangle$), so it will suffice to show that $\langle HP \rangle$ has growth at most $d$.  Quotienting by $H$, it suffices to show that $v_1,\dots,v_r$ generate a group of growth at most $d$.

Let $K = \langle HP \rangle/H$ be the group generated by the $v_1,\dots,v_r$.  This is a nilpotent group.
Inductively using the upper-triangular property of nilprogressions in $C$-normal form, we see that the $j^{th}$ term $K_j$ in the lower central series of $K$ is generated by a subset of $\{ v_i: N_i \geq C_{j,d,\eps}^{-1} n^{j/2} \}$ for some $C_{j,d,\eps}$ depending on $j,d,\eps$.  By the Bass-Guivarc'h formula \cite{bass, guivarch}, the order $D$ of growth of $K$ is then at most
$$ \sum_{j \geq 1} \# \{ v_i: N_i \geq C_{j,d,\eps}^{-1} n^{j/2} \} $$
and hence
\begin{align*}
 \prod_{i=1}^r N_i & \geq (C'_{d,\eps})^{-1} \prod_{j \geq 1} \prod_{1 \leq i \leq r: N_i \geq C_{j,d,\eps}^{-1} n^{j/2} } n^{1/2} \\ 
& \geq (C''_{d,\eps})^{-1} n^D 
\end{align*}
for some $C'_{d,\eps}, C''_{d,\eps}$ depending on $d,\eps$.  On the other hand, from the volume bound for nilprogressions in normal form and \eqref{happy} we have
$$ \prod_{i=1}^r N_i \leq C'''_{d,\eps} n^{d+1-\eps}$$
for some $C'''_{d,\eps}$ depending on $d,\eps$.  For $n$ large enough, we thus have $D \leq d$, and the claim follows.


\begin{thebibliography}{10}

\bibitem{bass}
H. Bass, \emph{The degree of polynomial growth of finitely generated nilpotent group}, Proceedings London Mathematical Society \textbf{25} (1972).

\bibitem{benson}
M. Benson, \emph{Growth series of finite extensions of $\Z^n$ are rational}, Invent. Math. \textbf{73} (1983), no. 2, 251--269. 

\bibitem{benson2}
M. Benson, \emph{On the rational growth of virtually nilpotent groups}, Combinatorial group theory and topology (Alta, Utah, 1984), 185--196, 
Ann. of Math. Stud., 111, Princeton Univ. Press, Princeton, NJ, 1987. 


\bibitem{b}
E. Breuillard, \emph{Local limit theorems and equidistribution of random walks on the Heisenberg group}, Geom. Funct. Anal. \textbf{15} (2005), no. 1, 35--82. 


\bibitem{bggt}
E. Breuillard, B. Guralnick, B. Green, T. Tao, \emph{Expansion in finite simple groups of Lie type}, J. Eur. Math. Soc. (JEMS) \textbf{17} (2015), no. 6, 1367--1434.


\bibitem{bgt}
E. Breuillard, B. Green, T. Tao, \emph{The structure of approximate groups}, Pub. Math. de l'IH\'ES \textbf{116} (2012), 116--221.

\bibitem{bgt-survey}
E. Breuillard, B. Green, T. Tao, \emph{Small doubling in groups}, Erd\H os centennial, 129–151, Bolyai Soc. Math. Stud., 25, J\'anos Bolyai Math. Soc., Budapest, 2013.

\bibitem{donne}
E. Breuillard, E. Le Donne, \emph{On the rate of convergence to the asymptotic cone for nilpotent groups and subFinsler geometry}, Proc. Natl. Acad. Sci. USA \textbf{11}0 (2013), no. 48, 19220--19226. 

\bibitem{bt}
E. Breuillard, M. Tointon, \emph{Nilprogressions and groups with moderate growth}, preprint.

\bibitem{freiman}
G. A. Fre{\v i}man, Foundations of a structural theory of set addition.  Translated from the Russian. Translations of Mathematical Monographs, Vol 37. American Mathematical Society, Providence, R. I., 1973.

\bibitem{green-ruzsa}
B. Green, I. Ruzsa, \emph{Freiman's theorem in an arbitrary abelian group}, J. Lond. Math. Soc. (2) \textbf{75} (2007), no. 1, 163--175.

\bibitem{gromov}
 M. Gromov, \emph{Groups of polynomial growth and expanding maps}, Inst. Hautes Etudes Sci. Publ. Math., \textbf{53} (1981), 53--73.

\bibitem{guivarch}
Y. Guivarc'h, \emph{Groupes de Lie \'a croissance polynomiale}, C. R. Acad. Sci. Paris S\'er. A–B \textbf{272} (1971).


\bibitem{john}
F. John, \emph{Extremum problems with inequalities as subsidiary conditions}, Studies and Essays Presented to R. Courant on his 60th Birthday, January 8, 1948, 187--204. Interscience Publishers, Inc., New York, N. Y., 1948. 

\bibitem{kh}
A. G. Khovanski{\v i}, \emph{The Newton polytope, the Hilbert polynomial and sums of finite sets}, Funktsional. Anal. i Prilozhen. \textbf{26} (1992), 57--63, 96.

\bibitem{malcev}
A. Mal'cev, \emph{On a class of homogeneous spaces}, Izvestiya Akad. Nauk SSSR, Ser Mat. \textbf{13} (1949), 9--32.

\bibitem{nsw}
A. Nagel, E. Stein, S. Wainger, \emph{Balls and metrics defined by vector fields I. Basic properties}, Acta. Math. \textbf{155} (1985), 103--147.

\bibitem{nat}
M. Nathanson, \emph{Growth of sumsets in abelian semigroups}, Semigroup Forum \textbf{61} (2000), 149--153.

\bibitem{nathanson}
M. Nathanson, I. Ruzsa, \emph{Polynomial growth of sumsets in abelian semigroups}, J. Th\'eor. Nombres Bordeaux \textbf{14} (2002), no. 2, 553--560. 

\bibitem{nguyen}
H. Nguyen, V. Vu, \emph{Optimal inverse Littlewood-Offord theorems}, Adv. Math. \textbf{226} (2011), no. 6, 5298--5319. 

\bibitem{nguyen-survey}
H. Nguyen, V. Vu, \emph{Small ball probability, inverse theorems, and applications}, Erd\H{o}s centennial, 409–463, Bolyai Soc. Math. Stud., 25, J\'anos Bolyai Math. Soc., Budapest, 2013. 


\bibitem{sanders}
T. Sanders, \emph{On a non-abelian Balog-Szemer\'edi-type lemma}, J. Aust. Math. Soc. \textbf{89} (2010), no. 1, 127--132.


\bibitem{sanders-survey}
T. Sanders, \emph{The structure theory of set addition revisited}, Bull. Amer. Math. Soc. (N.S.) \textbf{50} (2013), no. 1, 93--127. 

\bibitem{stoll}
M. Stoll, \emph{On the asymptotics of the growth of $2$-step nilpotent groups}, J. London Math. Soc. (2) \textbf{58} (1998), no. 1, 38--48. 

\bibitem{tao-product}
T. Tao, \emph{Product set estimates for non-commutative groups}, Combinatorica, \textbf{28} (2008), 547--594.

\bibitem{tao-vu}
T. Tao, V. Vu, \emph{John-type theorems for generalized arithmetic progressions and iterated sumsets}, Adv. Math. \textbf{219} (2008), no. 2, 428--449. 

\bibitem{tiep}
P. Tiep, V. Vu, \emph{Non-abelian Littlewood-Offord inequalities}, preprint.

\bibitem{tointon}
M. C. H. Tointon, \emph{Freiman's theorem in an arbitrary nilpotent group}, Proc. London Math. Soc. \textbf{109} (2014), 318--352.

\bibitem{wt}
T. Tao, J. Wright, \emph{$L^p$ improving bounds for averages along curves}, J. Amer. Math. Soc. \textbf{16} (2003), no. 3, 605--638. 

\bibitem{woess}
W. Woess, \emph{Random walks on infinite graphs and groups}, Cambridge Tracts in Mathematics, \textbf{138}. Cambridge University Press, Cambridge, 2000.

\end{thebibliography}
\end{document}